\documentclass[11pt,article]{amsart}
\usepackage[utf8]{inputenc}
\usepackage{enumerate}
\usepackage{amsmath, amssymb,amsthm}
\usepackage{mathtools}
\usepackage{leftidx}
\usepackage{cancel}
\usepackage{soul}
\usepackage[left=2.5cm, right=2.5cm, bottom=3cm]{geometry}
\usepackage[normalem]{ulem}
\numberwithin{equation}{section}
\usepackage{hyperref}
\usepackage{xcolor}
\definecolor{my-black}{rgb}{0,0,0}
\definecolor{my-blue}{rgb}{0,0,0.8}
\definecolor{my-red}{rgb}{0.8,0,0} 
\definecolor{my-green}{rgb}{0,0.5,0}
\hypersetup{colorlinks,
    linkcolor	= {my-blue},
    citecolor	= {my-red},
    urlcolor		= {my-black}
}
\usepackage{graphicx}
\usepackage{tikz-cd}

\theoremstyle{plain} 

\newtheorem{lemma}{Lemma}[section]
\newtheorem{theorem}{Theorem}
\newtheorem{corollary}{Corollary}[section]
\newtheorem{proposition}{Proposition}[section]

\newtheorem{question}{Question} 

\newtheorem*{theorem*}{Theorem} 

\theoremstyle{definition} 
\newtheorem{remark}{Remark}[section]

\newtheorem{definition}{Definition}[section]

\newtheorem{claim}{Claim}[section]

\theoremstyle{remark}
\newtheorem*{remark-non}{Remark}

\DeclareFontFamily{U}{mathx}{}
\DeclareFontShape{U}{mathx}{m}{n}{<-> mathx10}{}
\DeclareSymbolFont{mathx}{U}{mathx}{m}{n}
\DeclareMathAccent{\widehat}{0}{mathx}{"70}
\DeclareMathAccent{\widecheck}{0}{mathx}{"71}

\DeclareMathOperator{\supp}{supp}

\DeclareMathOperator{\sgn}{sgn}

\newcommand{\R}{\mathbb{R}}

\newcommand{\C}{\mathbb{C}}
\newcommand{\N}{\mathbb{N}}

\newcommand{\Z}{\mathbb{Z}}

\newcommand{\avgI}{\sout{1}}

\newcommand{\Fd}{\mathcal{F}_d}

\title[Perturbed lattice crosses and Heisenberg Uniqueness Pairs]{Perturbed lattice crosses\\ and Heisenberg Uniqueness Pairs}
\author{Danylo Radchenko}
\address{Laboratoire Paul Painlevé, University of Lille, France.}
\email{danradchenko@gmail.com} 

\author{Jo{\~a}o P.G. Ramos}
\address{Instituto de Matem\'atica Pura e Aplicada (IMPA), Rio de Janeiro, Brazil.}
\email{joao.ramos@impa.br}


\begin{document}
\begin{abstract}
{ This work focuses on several questions raised by H. Hedenmalm and A. Montes-Rodr\'iguez \cite{Hedenmalm-Montes-Annals} on Heisenberg Uniqueness Pairs for perturbed lattice crosses.} 

The first of them deals with a \emph{complete} characterization of $\beta>0$ for which, for a fixed $\theta \in \R,$ the translated lattice cross $\Lambda_{\beta}^{\theta} = ((\Z + \{\theta\}) \times \{0\}) \cup (\{0\} \times \beta \Z)$ satisfies that $(\Gamma,\Lambda_{\beta}^{\theta})$ is a Heisenberg Uniqueness Pair, where $\Gamma$ is the hyperbola in $\R^2$ with axes as asymptotes. As a matter of fact, we show that $(\Gamma,\Lambda_{\beta}^{\theta})$ is an H.U.P. \emph{if and only if} $\beta \le 1$, confirming a prediction made in \cite{Hedenmalm-Montes-Annals}.  

Furthermore, under modified decay conditions on the measures under consideration, we are able to prove \emph{sharp} results for when a perturbed lattice cross $\Lambda_{\bf A,B}$ is such that $(\Gamma,\Lambda_{\bf A,B})$ is an H.U.P. In particular, under such decay conditions, this solves another question posed by H. Hedenmalm and A. Montes-Rodr\'iguez. 

In addition, we give a negative answer to Open problem~(d) in
\cite{Hedenmalm-Montes-Annals}. For every lattice cross we construct a
nonzero smooth Klein--Gordon solution which vanishes on the cross and,
together with all its derivatives, belongs to \(C_0(\R^2)\). The
corresponding spectral measure is singular continuous. This shows that
the \(L^2\) condition in our positive results cannot be replaced by
planar continuity and decay. We also disprove the proposed weighted completeness
criterion in \cite{Hedenmalm-Montes-Annals}, and obtain a quantitative
sampling estimate in \(\mathcal H_1\)
for subcritical perturbed lattice crosses.

Our techniques run through the analysis of the action of the operator which maps the Fourier transform of an $L^1$ function~$\psi$ to the Fourier transform of $t^{-2} \psi(1/t)$. In other words, we analyze the operator taking the restriction to the $x$-axis of a solution $u$ to the Klein-Gordon equation to its restriction to the $y$-axis. This operator turns out to be related to the action of the Four-dimensional Fourier transform on radial functions, which enables us to use the philosophical framework and techniques of discrete uncertainty principles for the Fourier transform. 
\end{abstract}

\maketitle 

\section{Introduction}

\subsection{Historical Background} Let $f:\R \to \C$ be a measurable, square-integrable function. A crucial question in Fourier analysis, with consequences in fields ranging from physics to signal processing is: how well can one recover the information of such a function $f$ and its Fourier transform~$\widehat{f}$, given one only has \emph{partial} access to such information? 

As a first partial answer to such a question, one may invoke the classical \emph{Heisenberg uncertainty principle}. It guarantees that any function $f \in L^2$ cannot be too concentrated in space, unless it is the case that its Fourier transform $\widehat{f}$ is `spread out' in frequency: in particular, the mass of any such $f$ cannot be `overwhelmingly concentrated' in a space ball of small radius around any $x_0 \in \R$, while at the same time the mass of $\widehat{f}$ being concentrated in another ball of small radius around $\xi_0 \in \R$. 

As striking and elegant as such a principle may seem, it represents only the beginning of a rich literature on the subject of uncertainty principles. Indeed, several results in the literature have been of instrumental nature in further understanding the relationship between space-frequency concentration and recovery of functions. 

We highlight, in particular: (i) \emph{the Shannon-Whittaker formula} \cite{Shannon,Whittaker}, which states that, for a function $f \in L^2$ whose Fourier transform is compactly supported on a compact interval, we may recover it by its values from a certain rescaling of the set of integers. This implies that, if a Fourier transform of a compactly supported function vanishes on certain rescaling of the integers, then it must \emph{vanish identically}; (ii) \emph{Hardy's Uncertainty Principle} \cite{HardyUP}, which shows that, if a function and its Fourier transform are pointwise bounded by a multiple of a certain low-variance Gaussian, then it must \emph{vanish identically}; (iii) and finally, \emph{the Amrein-Berthier-Benedicks theorem} \cite{Amrein-Berthier-1,Amrein-Berthier,Benedicks}, which states that, for any two measurable sets $A,B\subset \R$ of \emph{finite} measure, a function cannot have support in $A$ with its Fourier transform supported in $B$, unless it is \emph{identically zero}. 

In this manuscript, we shall be interested in uncertainty principles of a similar flavour, but which are related to a certain given \emph{partial differential equation}. In fact, all of the results mentioned above may be interpreted as suitable uniqueness results for the \emph{harmonic oscillator}: if we define, for a fixed $f \in L^2$, the function $\Phi_f(x,t)$ as being a solution to 
\begin{equation*}
        i\,\partial_t \Phi_f  = (-\Delta + \pi^2 |x|^2) \Phi_f, \text{ for } (x,t) \in \R \times \R, 
\end{equation*}
with $\Phi_f(x,0)  = f(x)  \text{ for } x \in \R,$ then one has that $\Phi_f(x,1/8)$ is equal to (a constant multiple of) the Fourier transform of $f$. Hence, the uncertainty-flavoured results mentioned above translate naturally as results about \emph{uniqueness of solutions} to the time-dependent harmonic oscillator above. 

In line with these results, we highlight two main lines of related work. The first of them is dedicated to extending the philosophical view of uncertainty principles as uniqueness results for certain partial differential equations as a more concrete device. Here, we mention, for instance, the celebrated works by L. Escauriaza, C. Kenig, C. Ponce and L. Vega \cite{Escauriaza2006,Escauriaza20081,Escauriaza2008,Escauriaza2010,Escauriaza2010sharp}, where the authors prove a sharp version of Hardy's uncertainty principle for general Schr\"odinger equations with a potential. In that same line of work, we also highlight the following articles \cite{Kehle-Ramos,Goncalves-Ramos,Kulikov-Oliveira-Ramos} and the references therein. 

The second line of work mentioned above is the main focus of this manuscript. In order to describe it, let first $\mu$ denote a { finite complex} Borel measure on $\R^2$. Suppose, moreover, that $\supp(\mu) \subset \Gamma_0$, where $\Gamma_0$ is some smooth curve in $\R^2$. { We also assume that $\mu$ is absolutely continuous with respect to arclength measure on $\Gamma_0$.} A set $\Lambda\subseteq\R^2$ is said to form a \emph{Heisenberg Uniqueness Pair} (or H.U.P., for shortness) with $\Gamma_0$ if, whenever $\mu$ above satisfies $\widehat{\mu}|_{\Lambda} = 0$, then we must have $\mu \equiv 0$. 

This definition as such may seem unrelated to the topic at hand of uniqueness results for partial differential equations, but, as a matter of fact, in most cases of curves $\Gamma_0$ explored in the literature so far this relationship is not only evident but also crucial. Indeed, we highlight here three main cases: 

\vspace{2mm}

\noindent I. \emph{The parabola $\Gamma = \{ (x,y) \colon y = x^2\}$.} In this case, any measure $\mu$ supported on $\Gamma$ yields through Fourier transform a solution to the free \emph{Schr\"odinger equation} $
    \partial_t \widehat{\mu} + \frac{i}{\pi} \partial_x^2 \widehat{\mu} = 0;$ 
    
    \vspace{2mm}
    
\noindent II. \emph{The circle $\Gamma = \{(x,y) \colon x^2 + y^2 =1\}$.} Here, one readily obtains that $\widehat{\mu}$ satisfies the \emph{Helmholtz equation} $\Delta \, \widehat{\mu} = \pi^2 \widehat{\mu};$ 

    \vspace{2mm}

\noindent III. \emph{The hyperbola $\Gamma = \{ (x,y) \colon xy = 1\}$.} Finally, here one has that, if $\supp(\mu) \subset \Gamma$, then $\widehat{\mu}$ satisfies the \emph{Klein-Gordon equation} $
    \partial_x \partial_y \widehat{\mu} + \pi^2 \widehat{\mu} = 0.$

    \vspace{2mm}

Note that, in light of these considerations, the first line of work described above \cite{Escauriaza2006,Escauriaza2008,Escauriaza20081,Escauriaza2010,Escauriaza2010sharp} is directly related to the concept of Heisenberg Uniqueness Pairs for the Schr\"odinger equation. In this work, on the other hand, we shall focus on the third case, in which Heisenberg Uniqueness Pairs are directly related to \emph{uniqueness sets} for the Klein-Gordon equation. We refer the reader to \cite{Escauriaza2008,Grochenig-Jaming,Aingeru-etal,Jaming-Kellay,Lev,Sjolin1,Sjolin2,Vieli1,Ghobber-Jaming} for references in the directions of the first two other instances described above. 

In spite of the fact that the definition of Heisenberg Uniqueness Pairs becomes natural in the context introduced above, it seems, to the best of our knowledge, that it has first been explicitly introduced only recently in the seminal work of H.~Hedenmalm and A.~Montes-Rodriguez~\cite{Hedenmalm-Montes-Annals}. There, besides introducing the concept and discussing some basic examples similar to the ones above, they \emph{completely resolve} the question of determining all $\alpha,\beta > 0$ such that the pair $(\Gamma,\Lambda_{\alpha,\beta})$ is a Heisenberg Uniqueness Pair, where $\Gamma$ denotes the hyperbola $\Gamma = \{(x,y) \in \R^2 \colon xy = 1\},$ and $\Lambda_{\alpha,\beta} = ((\alpha \Z) \times \{0\}) \cup (\{0\} \times (\beta \Z))$ denotes the \emph{lattice cross} with dilation parameters $\alpha,\beta.$ 

More specifically, they showed that $(\Gamma,\Lambda_{\alpha,\beta})$ is an H.U.P. if and only if $\alpha \beta \le 1.$ Their result, in addition to possessing a beautiful proof using elements from dynamical systems in novel ways, drew new connections between fields in analysis and raised several important questions in the direction of Fourier uncertainty principles. Effectively, several subsequent works further exploited the connections between such problems and the Klein-Gordon equation such as \cite{Hedenmalm-Monte-Klein-Gordon, CHMR, Hedenmalm-Montes-Hinfty,Bakan-Hedenmalm-Montes-Radchenko-Viazovska,BHMRRV-second,Ghosh-Giri}, and, as recently highlighted \cite{Ramos-Stoller,Bakan-Hedenmalm-Montes-Radchenko-Viazovska,BHMRRV-second}, Heisenberg Uniqueness Pairs for the hyperbola turn out to have a connection to recent \emph{Fourier interpolation results}, such as the interpolation formula due to the first author and M. Viazovska \cite{Radchenko-Viazovska}, which possesses a particularly important connection to the solution of the sphere packing problem in dimensions 8 and 24 \cite{CKMRV19,Viazovska,cohn2022universal}.  

\subsection{Main results} In this manuscript, we shall focus on a question initially raised by H. Hedenmalm and A. Montes-Rodr\'iguez: 

\begin{question}[Open~problem~(a) in \cite{Hedenmalm-Montes-Annals}]\label{q:Hedenmalm-Montes-original} 
If one defines the \emph{translated lattice crosses} 

\begin{equation}\label{eq:lattice-cross}
\Lambda_{\beta}^{\theta} = ((\Z + \{\theta\}) \times \{0\}) \cup (\{0\} \times \beta \Z),
\end{equation}
for which $\beta > 0$ is $(\Gamma,\Lambda_{\beta}^{\theta})$ a Heisenberg Uniqueness Pair?  

Moreover, if one considers suitable perturbations $\tilde{\Lambda} \subset \left( \R \times \{0\} \right) \cup \left( \{0\} \times \R \right)$ of lattice crosses, when is $(\Gamma,\tilde{\Lambda})$ a \emph{H.U.P.}? 
\end{question} 

The authors conjecture in \cite{Hedenmalm-Montes-Annals} that the methods developed there could be adapted to prove results for the first of those questions. On the other hand, the perturbative regime suggested by them seemed to require new ideas outside of the scope of their techniques, and it was left open. Regarding the latter question, we note that the only result that we are aware of in this direction was by the second author and M. Stoller \cite{Ramos-Stoller}, which showed that, at least for certain smooth classes of measures, it is indeed the case that one is able to perturb $\Lambda_{\alpha,\beta}$ slightly and still obtain positive results in the case when $\alpha \cdot \beta = 1$. 

Regarding the first question on lattice crosses translated by a fixed quantity $\theta \in \R$, D. Giri and R. Rawat \cite{Giri-Rawat} were the first ones to consider it explicitly after the original comment in \cite{Hedenmalm-Monte-Klein-Gordon}. In their manuscript, it was first claimed that $(\Gamma,\Lambda_{\beta}^{\theta})$ is a Heisenberg uniqueness pair if and only if $\beta \le p.$ 

Although the examples constructed in \cite{Giri-Rawat} show that $\beta \le p$ is a necessary condition, the ``if'' part of this result turned out to be false, as noted by the authors themselves in \cite{Giri-Rawat-corrigendum}, due to a mistake in a computation undertaken on \cite[page~8]{Giri-Rawat}, while defining their main operators in analogy to the original work of Hedenmalm and Montes-Rodr\'iguez. 

The main results in this manuscript are inspired by and provide answers to Question \ref{q:Hedenmalm-Montes-original}. Indeed, the first result we highlight here deals with the context of translated lattice crosses. Effectively, we are able to show that, independently of how $\theta \in \R$ is chosen, $(\Gamma,\Lambda_{\beta}^{\theta})$  is a Heisenberg Uniqueness Pair (H.U.P.) \emph{if and only if} $\beta \le 1$, settling hence the first part of Question \ref{q:Hedenmalm-Montes-original}. 

\begin{theorem}\label{thm:main} Let $\Gamma = \{(x,y) \in \R^2 \colon xy=1\}$ denote the hyperbola, and consider the translated lattice crosses $\Lambda_{\beta}^{\theta}$ as defined in \eqref{eq:lattice-cross}, for some arbitrary $\theta \in \R$. Then $(\Gamma,\Lambda_{\beta}^{\theta})$ is a Heisenberg Uniqueness Pair if and only if $\beta \le 1.$ 
\end{theorem}

Before describing the proof of Theorem \ref{thm:main}, we remark on its surprisingly challenging nature by comparing our result with the ones obtained by using the techniques available in the literature of H.U.P.'s for the hyperbola so far. 

Indeed, in spite of the fact that the original prediction in \cite{Hedenmalm-Montes-Annals} is, after all, correct, the techniques needed in order to fully accomplish it differ significantly from the original ones. This is especially highlighted by the previously available results in this direction in the literature: in order to prove the fact that $\beta \le 1$ implies that $(\Gamma,\Lambda_{\beta}^{\theta})$ is indeed a H.U.P., one needs to resort to the subsequent work \cite{Hedenmalm-Monte-Klein-Gordon}, which in spite of being close in spirit to the results in \cite{Hedenmalm-Montes-Annals}, was only developed several years later. For the proof that $\beta \le 1$ is also \emph{necessary}, the differences are even more drastic: as a matter of fact, the counterexamples in \cite{Hedenmalm-Montes-Annals} are functions of the form 
\begin{equation}\label{eq:Poisson-diff}
\psi(t) = P_{z_1}(t) - P_{z_2}(t),
\end{equation}
where $P_z(t) = \frac{y}{(x-t)^2 + y^2}$ is the Poisson kernel at $z = x + i y$, and $z_1,z_2 \in \C_+$. In \cite{Giri-Rawat}, the authors, while employing the same Poisson extension techniques, obtain that $\beta \le p$ is a necessary condition, whenever $\theta = \frac{1}{p}, \, p \in \N$. Since this lower bound on $\beta$ blows up as $p \to \infty$, the gap between the necessary and sufficient conditions converges to infinity as the parameter $\frac{1}{p} \to 0$, which is diametrically opposite to what happens when $\theta = 0$.  

In order to circumvent these issues, we need two main ingredients for the proof of Theorem \ref{thm:main}: the first is a corrected proof of the ``if'' part of the result from~\cite{Giri-Rawat}, which, as mentioned above, can be obtained by employing the techniques from \cite{Hedenmalm-Monte-Klein-Gordon}, as noted in the corrigendum \cite{Giri-Rawat-corrigendum}. 

The second main ingredient, and the main novel feature of this work, is a new set of counterexamples for the ``only if'' part, going beyond the ones given by \eqref{eq:Poisson-diff}. These new counterexamples are inspired by the correspondence between Heisenberg Uniqueness Pairs and Fourier uncertainty principles highlighted in \cite{Bakan-Hedenmalm-Montes-Radchenko-Viazovska}, together with recent constructions of A.~Kulikov, F.~Nazarov and M.~Sodin \cite{Kulikov-Nazarov-Sodin} of Fourier non-uniqueness sets in dimension~1. 

In order to lay the foundation for an iterative process, as done in \cite{Kulikov-Nazarov-Sodin}, we need to consider a different setup, arising from the operator $T$ which takes as input a function $f$ and returns the restriction of the solution $u$ to the Klein-Gordon equation $\partial_x \partial_y u + \pi^2 u = 0$ with $u|_{\{y=0\}} = f$ to the $y$-axis. An equivalent way to define this operator is by letting it be the operator which takes as input the Fourier transform of a function~$\psi$, and returns as an output the Fourier transform of the function $t^{-2} \psi(1/t)$. As we shall see, this operator is, as in \cite{Bakan-Hedenmalm-Montes-Radchenko-Viazovska}, related to certain four-dimensional Fourier transforms of radial functions. 

The crucial difference here is that its definition allows us to bypass radiality, making it possible for one to consider asymmetric sets in the counterexample construction - a feature not achiavable through a purely Fourier-based approach. Moreover, the properties of the operator $T$ (see Proposition \ref{prop:properties-T}) show that the iteration process has to be done in two separate parts, according to the positive and negative parts of the zero set under consideration. This allows us to maintain the crucial fact that the operator $T$ takes functions supported on a half-line to functions supported on the opposite half-line, allowing one to deduce results also in the case of one-sided H.U.P.'s -- as highlighted in Section \ref{sec:one-sided}. 

\vspace{2mm}

Our next result deals, on the other hand, with the topic of H.U.P.'s for perturbed lattice crosses, as stated in the second part of Question \ref{q:Hedenmalm-Montes-original}. Even in the case dealt with in \cite{Ramos-Stoller}, the perturbations considered are taken to be origin-symmetric and considerably close to the original lattice cross points, with the additional constraint that the functions considered there are taken to be \emph{real-valued}, which excludes a large variety of cases. In that context, we give below the first instance of \emph{sharp} conditions in order for such a perturbed lattice cross to form a H.U.P. with the hyperbola, under different decay assumptions when compared to the ones in, for instance, \cite{Hedenmalm-Montes-Annals,Hedenmalm-Monte-Klein-Gordon}. 

In order to state such a result, we let first
    $$\mathcal{H}_{\ell}(\R) =\{g:\R \to \C \colon g  \in L^2((1+|y|^{2\ell})\, dy), \, \widehat{g} \in C_0(\R)\},$$ 
where $\ell \in \Z_{\ge 0}.$ We remark that these spaces are somewhat reminiscent of the ones defined in \cite{Goncalves-Ramos} in the context of H.U.P.'s for the parabola.  

\begin{definition} Let $\mu$ be a measure on $\R^2$ supported on $\Gamma$ which can be represented as $d\mu(t,1/t) = \psi(t) \, dt,$ for some $\psi \in \mathcal{H}_{\ell}(\R)$, and let $\Lambda$ be a subset of $(\R \times \{0\} ) \cup (\{0\} \times \R)$. We say that the pair $(\Gamma,\Lambda)$ is an \emph{$\mathcal{H}_{\ell}$-Heisenberg Uniqueness Pair} (or $\mathcal{H}_{\ell}$-H.U.P.) if for any $\mu$ as above with $\widehat{\mu}|_{\Lambda} = 0,$ we have $\mu \equiv 0.$ 
\end{definition} 
Note that, if $\ell \ge 1$, then $\mathcal{H}_{\ell} \subset L^1(\R)$, and hence the concept of $\mathcal{H}_{\ell}$-H.U.P.'s is well-defined in that case, since $\widehat{\mu}(x,0) = \widehat{\psi}(x)$, and $\widehat{\mu}(0,y) = \widehat{\varphi}(y)$, where $\varphi(t)= \frac{1}{t^2} \psi\left(\frac{1}{t}\right) \in L^1$. For the $\ell = 0$ case, note that $\psi \in \mathcal{H}_0$ is equivalent to $\varphi \in L^2(t^2\,dt)$, which means that the Fourier transform of $t\cdot \varphi(t)$ is an $L^2$ function. By standard distribution theory arguments (see Section \ref{sec:prelim} below), this implies that the Fourier transform of $\varphi$ as a distribution may be identified with a \emph{continuous} function of polynomial growth, and hence $\widehat{\mu}|_{\left(\R \times \{0\}\right) \cup \left( \{0\} \times \R\right)}$ is defined pointwise and continuous, which shows that the concept of $\mathcal{H}_0$-H.U.P.'s is also well-defined. 

With that in mind, we shall use the following notation for the \emph{generalized lattice cross} 
\[
\Lambda_{{\bf A,B}} = (A \times \{0\}) \cup (\{0\} \times B),
\]
where $A = \{a_n\}_{n \in \Z}, \, B = \{b_n\}_{n \in \Z}$ are two separated sequences of points on $\R$, in the sense that $$\min(|a_{n+1} - a_n|,|b_{n+1} - b_n|) > \delta, \,\qquad \forall \, n \in \Z,$$
for some $\delta > 0$. We will sometimes also call such sets \emph{uniformly discrete}.  

\begin{theorem}\label{thm:improved} Let $\Gamma = \{(x,y) \in \R^2 \colon xy = 1\}$ denote the hyperbola. The following assertions hold: 
\begin{enumerate}[\normalfont(I)]
\item Suppose that 
$$\sup_n |a_{n+1} - a_n| = \alpha, \, \, \sup_n |b_{n+1} - b_n| = \beta.$$
Then the pair $(\Gamma, \Lambda_{\bf A,B})$ is an $\mathcal{H}_0$-Heisenberg Uniqueness Pair if $\alpha \beta < 1$.  Furthermore, $(\Gamma,\Lambda_{\bf A,B})$ is an $\mathcal{H}_2$-\emph{H.U.P.} if $\alpha \beta = 1$. 

\item Suppose that 
$$ \inf_n|a_{n+1} - a_n| = \alpha, \, \, \inf_n |b_{n+1} - b_n| = \beta,$$ 
with $\alpha \beta > 0.$ Then $(\Gamma,\Lambda_{\bf A,B})$ is \emph{not} an $\mathcal{H}_1$-\emph{H.U.P.} (and hence also \emph{neither a } $\mathcal{H}_0$-\emph{H.U.P., nor a  regular H.U.P.}) if $\alpha \beta > 1$. 
\end{enumerate}
\end{theorem}

In order to prove the positive part of this result, we use a \emph{Poincar\'e-Wirtinger inequality} argument: by using such estimates, one can essentially control the $L^2$-norm of $x\cdot \psi(x)$, where $\psi$ is the density of a measure $\mu$ as above with $\widehat{\mu}|_{\Lambda_{\bf A,B}} = 0$, \emph{strictly} from below by the $L^2$ norm of $\psi$ itself. On the other hand, the same argument can be used for $\varphi(t) = t^{-2} \psi(1/t)$, which implies the reverse \emph{strict} inequality, yielding a contradiction.

The sketch of proof above is inspired by the positive Fourier Uniqueness results from \cite{Kulikov-Nazarov-Sodin}, but again with differences induced by the context of the result. It is striking to note that, by considering the definition of a $\mathcal{H}_0$-H.U.P., then the argument carried out here allows for a considerable simplification of the subcritical results from \cite{Hedenmalm-Montes-Annals}. It is further important to notice that the class of functions $\mathcal{H}_0$ is inherently \emph{different} from the ones considered in \cite{Hedenmalm-Montes-Annals,Hedenmalm-Monte-Klein-Gordon,CHMR}, which implies that, in the case when $A = \alpha \Z, B = \beta \Z, \, \alpha \beta \le 1,$ our results actually \emph{extend} the ones in those references to different classes. We refer the reader to Sections \ref{sec:neg} and \ref{sec:alt-psi-proof} below for more details.

{ 

The subcritical assertion in Theorem \ref{thm:improved} admits a
quantitative version in the class \(\mathcal H_1\). Let
\(\psi\in\mathcal H_1(\R)\), and put
\[
 \varphi(t)=t^{-2}\psi(1/t),\qquad t\ne0.
\]
Both \(\widehat\psi\) and \(\widehat\varphi\) are continuous and
vanish at infinity. Thus their values at the two arms of a generalized
lattice cross are well-defined.

\begin{theorem}\label{thm:quantitative-sampling}
Let \(A=\{a_n\}_{n\in\Z}\) and \(B=\{b_n\}_{n\in\Z}\) be strictly
increasing, uniformly discrete sequences, unbounded in both directions.
Assume that
\[
 \sup_{n\in\Z}(a_{n+1}-a_n)\leq\alpha,
 \qquad
 \sup_{n\in\Z}(b_{n+1}-b_n)\leq\beta,
 \qquad \alpha\beta<1.
\]
Then every \(\psi\in\mathcal H_1(\R)\) satisfies
\begin{equation}\label{new-eq:quantitative-bound}
 \begin{split}
 \beta\|\psi\|_2^2+\alpha\|t\psi\|_2^2
 \leq {}&\frac{\alpha\beta}{(1-\sqrt{\alpha\beta})^2}\\
 &{}\times\left(
 \sum_{n\in\Z}|\widehat\psi(a_n)|^2+
 \sum_{n\in\Z}|\widehat\varphi(b_n)|^2\right).
 \end{split}
\end{equation}
\end{theorem}

A recent result of Bondarenko and Seip
\cite[Theorem~1.1]{Bondarenko-Seip-dense-zeros} gives a complementary indication
that the Sobolev assumption cannot in general be replaced by continuity
and integrability. For every \(0\leq\gamma\leq1/2\), they construct a
nonzero real-valued even function \(f_\gamma\in C(\R)\cap L^1(\R)\cap
L^2(\R)\) satisfying \(\mathcal F_1(f_\gamma)=f_\gamma\), with
\(\mathcal F_1\) normalized as in \eqref{eq:fourier-d-def}, and vanishing at
\[
 \lambda_n(\gamma)=\frac{\sqrt n}{[\log(e+n)]^\gamma},
 \qquad n\geq0.
\]
By evenness, \(f_\gamma\) vanishes on
\(\Lambda_\gamma=\{\pm\lambda_n(\gamma):n\geq0\}\). For
\(\gamma>0\),
\[
 \lambda_n(\gamma)
 \big(\lambda_{n+1}(\gamma)-\lambda_n(\gamma)\big)\longrightarrow0.
\]
Thus \((\Lambda_\gamma,\Lambda_\gamma)\) lies in the uniqueness range
of the Kulikov--Nazarov--Sodin theorem
\cite[Theorem~1.3-UP]{Kulikov-Nazarov-Sodin} for \(H^1\) functions
\(h\) satisfying \(h,\mathcal F_1h\in W^{1,2}(\R)\),
although
uniqueness fails in the larger class above. This result concerns
ordinary Fourier uniqueness rather than lattice-cross sampling, but it
reinforces the role of derivative control in
Theorem~\ref{thm:quantitative-sampling}.
}

For Part (II) of Theorem \ref{thm:improved}, we will use essentially the same techniques used in order to construct the counterexamples in the ``only if'' part of Theorem \ref{thm:main}. For that reason, and for the inherent interest in such a construction in its own right, we state an unified result concerning the counterexamples constructed.

\begin{theorem}\label{thm:negative-gen} Let $A = \{a_n\}_{n \in \Z}$ and $B = \{b_n\}_{n \in \Z}$ be two sequences of real numbers such that 
\[
\min\left( \liminf_{n \to \pm \infty} |a_{n+1}-a_n|, \liminf_{n \to \pm \infty} |b_{n+1} -b_n|\right) > 1.
\]
Then there exists a function $\psi \in L^1(\R)$ such that $\psi \not \equiv 0$ and 
\begin{equation}\label{eq:HUP-cond-gen}
\int_{\R} e^{-\pi i \tau a_n} \psi(\tau) \, d\tau = \int_{\R} e^{-\pi i b_n/\tau} \psi(\tau) \, d \tau = 0, \quad 
\forall \, n \in \Z. 
\end{equation} 
\end{theorem}

We further note that the results in Theorem \ref{thm:improved} can also be strengthened so that the density condition in its positive part are only taken to hold in a limit sense. The proof of such a result, as expected, runs through a generalization of the results in \cite{Kulikov-Nazarov-Sodin} to higher dimensions. We defer the discussion of this to Section \ref{sec:KNS}. 

Lastly, we briefly point out that our methods allow for a \emph{complete} solution of the original problems studied by H. Hedenmalm and A. Montes-Rodr\'iguez in \cite{Hedenmalm-Montes-Annals}, under the $\mathcal{H}_0$-condition, even in the critical $\alpha \beta = 1$ case. It also yields a solution of the $\mathcal{H}_0$-modification of Theorem \ref{thm:main}. We briefly highlight these facts in the following result (see Section \ref{sec:alt-psi-proof} for further details): 

\begin{theorem}\label{thm:L^2-HMR} Let $\Gamma$ be the hyperbola as defined in Theorem \ref{thm:main}, and $\theta \in \R$. If $\Lambda_{\alpha,\beta}^{\theta} = \left( \left( \alpha \Z + \{\theta\}\right) \times \{0\} \right) \cup \left( \{0\} \times \beta \Z\right)$, then $(\Gamma,\Lambda_{\alpha,\beta}^{\theta})$ is a $\mathcal{H}_0$-{\emph{H.U.P.}} if, and only if, $\alpha \beta \le 1$. 
\end{theorem}

The proof of that result employs similar techniques as that of Theorems \ref{thm:main} and \ref{thm:improved} in order to deal with the subcritical and supercritical cases. In the critical case, however, a more subtle analysis is needed, and the equality case in the Poincar\'e--Wirtinger inequality yields that any function yielding a candidate to fail the H.U.P. property satisfies two different kinds of \emph{periodicity} conditions. One of them is a usual $2$-periodicity condition, and the other is that the function composed with $1/t$ is still $2$-periodic. By using classical results on the geometry of discrete groups, we are able to conclude that the only functions satisfying those conditions are constants, and a direct analysis shows that the only alternative is that the functions we deal with are identically zero.

{ 
The last results also have a direct bearing on other open problems from
\cite{Hedenmalm-Montes-Annals}. Theorem \ref{thm:L^2-HMR} gives
uniqueness under a condition on the density of the spectral measure, hence it is natural to ask whether this condition can be replaced by a
condition on the resulting solution alone. The following theorem shows
that continuity and vanishing at infinity do not suffice, which negatively answers Open problem (d) in \cite{Hedenmalm-Montes-Annals}. 

\begin{theorem}\label{thm:continuous-counterexample}
Let \(\Gamma=\{(t,t^{-1}):t\in\R\setminus\{0\}\}\). For every
\(\alpha,\beta>0\), there is a nonzero real-valued function
\(u\in C^\infty(\R^2)\) such that
\[
 \partial_x\partial_y u+\pi^2u=0,
 \qquad
 u(\alpha n,0)=u(0,\beta n)=0\quad(n\in\Z),
\]
and
\[
 \partial_x^j\partial_y^k u\in C_0(\R^2)
 \qquad(j,k\in\Z_{\geq0}).
\]
Moreover, \(u\) is the Fourier transform of a finite real signed
singular continuous measure supported on \(\Gamma\). For each fixed
cross, the space of such solutions is infinite dimensional.
\end{theorem}

Canto-Mart\'in, Hedenmalm, and Montes-Rodr\'iguez had identified this
possibility in \cite[Section~4.3]{CHMR}: a singular continuous measure
on \(\Gamma\), with Fourier transform vanishing at infinity and on the
critical lattice cross, would give a negative answer to Open
problem~(d). Theorem~\ref{thm:continuous-counterexample} supplies such
measures for every lattice cross and gives decay after arbitrary
differentiation.

We also obtain a counterexample to the weighted completeness question
in \cite[Section~7(b)]{Hedenmalm-Montes-Annals}.

\begin{theorem}\label{prop:weighted-counterexample}
For every \(1<p<\infty\), let \(M(t)=(1+t^2)^{-p}\). Then
\[
 \int_{\R}\frac{dt}
 {(1+t^2)^{p/(p-1)}M(t)^{1/(p-1)}}=\infty.
\]
Nevertheless, for every \(\beta>1\), the families
\[
 \{e^{\pi i nt}:n\in\Z\},\qquad
 \{e^{\pi i\beta n/t}:n\in\Z\}
\]
are not complete in \(L^p(\R,M(t)\,dt)\).
\end{theorem}

}

{ 
The manuscript is organized as follows. In Section~\ref{sec:prelim} we
introduce the operator \(T\) and the corresponding equivalences. In
Section~\ref{sec:proof-neg} we prove Theorem~\ref{thm:negative-gen},
which is used in the supercritical regime. Section~\ref{sec:counterex}
contains the proofs of Theorems~\ref{thm:main}, \ref{thm:improved}, and
\ref{thm:quantitative-sampling}. In Section~\ref{sec:applications} we
prove the asymptotic and one-sided variants. Its final subsection
contains the critical \(\mathcal H_0\) argument, the proof of
Theorem~\ref{thm:L^2-HMR}, and the proofs of
Theorems~\ref{thm:continuous-counterexample} and~\ref{prop:weighted-counterexample}.

\subsection*{LLM Usage} The preliminary versions of this manuscript were completely created by humans. Large Language Models only became important in the present version. Indeed, the second named author discovered, in parallel, Theorem \ref{thm:continuous-counterexample} by exploring with LLMs. 

More specifically, the second named author asked a Large Language Model to gauge whether Open problem (d) from \cite{Hedenmalm-Montes-Annals} was more likely to be true or false. The model recommended an attempt at falsifying it, defining a measure $\sigma = \nu - \iota_*\nu$ (see Section \ref{sec:applications} for more details), and the second named author suggested back a strategy similar as the one by Kulikov--Nazarov--Sodin \cite{Kulikov-Nazarov-Sodin} in order to construct the counterexample. As simply considering the map that takes $\nu \mapsto \iota_* \nu$ as an iteration device would likely lead to $\sigma$ being degenerate, other alternatives were considered. 

In light of the examples from \cite{Kulikov-Nazarov-Sodin}, the second author asked the model to modify this idea. It then suggested the construction of $\nu$ as done below, and the details have been refined as a cooperation between the LLM and the author. 

Aside from that result, Theorem \ref{prop:weighted-counterexample} and its proof were obtained by the model, and the proof of Theorem \ref{thm:quantitative-sampling} was also generated by cooperation between the second named author and the LLM.

The authors have thoroughly checked, rewritten and expanded the proofs obtained in cooperation with large language models. We take full responsibility for such results and their correctness, as well as we take responsibility for the other human-generated results. 
}

\vspace{3mm}

\subsection*{Data availability.} No additional data other than the proofs in this manuscript was produced or used. 

\vspace{3mm}

\subsection*{Conflict of interest statement.} We declare that we do not possess any conflict of interest. 

\section{Preliminaries}\label{sec:prelim}

\subsection{Notation and definitions} We shall use the following notation throughout the manuscript: whenever not further specified, $a \lesssim b$ means that $a \le C \cdot b$, where $C>0$ is a finite numerical constant whose exact value is not important for the respective context. \\

We define the Fourier transform  of $f \in L^1(\R^d)$ in dimension $d \in \N$ as 
\begin{equation}\label{eq:fourier-d-def}
\Fd(f)(\xi) = \int_{\R^d} f(x) e^{-2\pi i x \xi} \, dx. 
\end{equation}
In some cases, however, it will be convenient to consider a different normalization: 
\begin{equation}\label{eq:fourier-d-1-special}
\widehat{f}(\xi) = \frac{1}{\sqrt{2}^d} \int_{\R^d} f(x) e^{-\pi i x\cdot \xi } \, dx.
\end{equation}
This distinction will spare us from some unnecessary calculations with changes of variables. The inverse of the operator defined in \eqref{eq:fourier-d-1-special} will be denoted by 
\[
\widecheck{f}(\xi) = \frac{1}{\sqrt{2}^d} \int_{\R^d} f(x) e^{\pi i x \xi} \, dx.
\]
Finally, for a given class of measurable functions $\mathcal{C}$, we will usually write that $(\Gamma,\Lambda)$ is a $\mathcal{C}$-Heisenberg Uniqueness Pair (or $\mathcal{C}-$H.U.P.) if, given any $\mu$ supported on $\Gamma$ absolutely continuous with density $\psi \in \mathcal{C}$ and such that $\widehat{\mu}|_{\Lambda} = 0$, we have $\mu \equiv 0$. Here, note that the normalization employed is crucial in order to state an accurate version of Theorems \ref{thm:main} and \ref{thm:improved}. 

\subsection{Function spaces and embeddings} We shall use the following notation for \emph{homogeneous} Sobolev spaces: 
\begin{equation}
H^s(\R) = \left\{ f \in \mathcal{S}'(\R) \colon \, \widehat{f} \text{ is a function in } L^1_{loc}(\R\setminus\{0\}) \text{ and } \int_{\R} |\widehat{f}(\xi)|^2 |\xi|^{2s} \, d\xi < +\infty \right\},
\end{equation}
where $\mathcal{S}(\R)$ denotes the Schwartz class, and $\mathcal{S}'(\R)$ denotes the space of tempered distributions. In the case of a domain $\Omega \subset \R^d$, we will use the following notation for the \emph{inhomogeneous} Sobolev space: 
\[
W^{1,p}(\Omega) = \left\{ f:\Omega \to \C \colon \, \nabla f \text{ exists in the weak sense, and }\|f\|_p + \|\nabla f\|_p < +\infty\right\}. 
\] 
With regard to Sobolev spaces on intervals of the real line, we highlight the classical \emph{Poincar\'e-Wirtinger inequality}, which will be crucial in the proof of Theorem \ref{thm:improved}, as well as in Section \ref{sec:KNS}. We refer the reader to \cite[Section~7.7]{HLPIneq} for a proof of this result. 

\begin{theorem*}[Poincar\'e-Wirtinger] Let $(a,b) = I \subset \R$ be an open interval. Suppose $f \in W^{1,2}(I)$ is such that $f(a) = f(b) = 0$. Then we have that
\begin{equation}\label{eq:p-w-set} 
\int_I |f(t)|^2 \, dt \le \left(\frac{|I|}{\pi}\right)^2 \int_I |f'(t)|^2 \,dt,
\end{equation} 
with equality if, and only if, $f(x) = c \sin\left( \frac{\pi (x-a)}{|I|}\right)$ for some $c \in \C$. 
\end{theorem*} 

We also note the following embedding estimate, which will be particularly useful in Section \ref{sec:KNS}:  

\begin{proposition}\label{prop:embedding} There exists $C>0$ such that, for any $\theta \in L^2(\R) \cap H^1(\R)$ with $\theta(0) = 0$, we have 
\[
\|\theta\|_{L^2(dt/|t|)} \le C \|\theta\|_{L^2}^{1/2}\|\theta'\|_{L^2}^{1/2}. 
\]
\end{proposition}
\begin{proof} We begin by noting that $\theta \in H^1(\R)$ implies readily that $\theta \in AC(\R)$. Moreover, we have the estimate 
\[
|\theta(s)| \le \int_0^s |\theta'(t)| \, dt \le s^{1/2} \left(\int_0^s |\theta'(t)|^2 \, dt\right)^{1/2} \le \|\theta'\|_2 \cdot s^{1/2}. 
\]
Hence, we take $\delta > 0$ and split 
\begin{align*}
\|\theta\|_{L^2(dt/|t|)}^2 & = \int_{[-\delta,\delta]} \frac{|\theta(t)|^2}{|t|} \, dt + \int_{\R \setminus [-\delta,\delta]} \frac{|\theta(t)|^2}{|t|} \, dt \cr 
            & \le 2\delta \|\theta'\|_{L^2}^2 + \frac{1}{\delta} \|\theta\|_2^2. 
\end{align*} 
Choosing $\delta >0$ so as to minimize the right-hand side above directly implies the claim. 
\end{proof} 

Finally, we prove the fact mentioned in the introduction that $\psi \in \mathcal{H}_0$ implies that $\widehat{\varphi} \in C(\R)$. We remark first that, as previously noted, $\psi \in L^2$ implies $\varphi \in L^2(x^2 \, dx)$. This shows that the Fourier transform of $x \varphi(x)$ is $L^2$-integrable, which is equivalent to the assertion that $\widehat{\varphi} \in H^1(\R)$. The next result then proves the claim that $\widehat{\varphi} \in C(\R)$. 

\begin{proposition} Let $f \in H^1(\R)$. Then $f$ may be represented as a continuous function. 
\end{proposition}

\begin{proof} Let $f'$ denote the derivative of $f$ as a tempered distribution. Then we have that, for $g \in \mathcal{S}(\R),$
\[
\langle f',g\rangle = -\langle f,g'\rangle = - \langle \widehat{f}, \widehat{g'}\rangle. 
\]
Since $\widehat{f} \in L^1_{loc}(\R \setminus \{0\})$, we may write the pairing $\langle \widehat{f},\widehat{g'}\rangle = -\pi i \cdot \int_{\R} \widehat{f}(\xi) \, \xi\, \overline{\widehat{g}(\xi)} \, d\xi$, whenever $\widehat{g}$ has compact support not containing $0$. Since the latter integral is bounded by $\pi \|f\|_{H^1} \|g\|_2,$ we may actually extend this identification to all $g \in \mathcal{S}(\R).$ Hence, 
\[
\langle f',g\rangle = \int_{\R}(\pi i \xi) \widehat{f}(\xi) \, \overline{\widehat{g}(\xi)} \, d\xi = \langle \widecheck{\pi i \xi \cdot \widehat{f}}, g\rangle. 
\]
This shows that $\widehat{f'}$ is a locally integrable function, and $\widehat{f'}(\xi) = \pi i \xi \widehat{f}(\xi) \in L^2$, which shows that $f' \in L^2(\R)$. Now, let $h \in \mathcal{S}'(\R)$ denote another distribution such that $h' = f'$. We then have 
\[
\langle h-f,g'\rangle = - \langle h'-f',g\rangle = 0 
\]
for all $g \in \mathcal{S}(\R)$. On the other hand, 
\[
\langle h-f,g'\rangle = \langle \widehat{h-f},(\pi i \xi) \cdot \widehat{g} \rangle. 
\]
Let then $\eta$ be compactly supported, and $0 \not\in \supp(\eta)$. Then $\eta(\xi) = \pi i  \xi \widehat{g}(\xi)$ for some $g \in \mathcal{S}(\R)$, and hence $\langle \widehat{h-f},\eta\rangle = 0$. Thus, we conclude that $\widehat{h-f}$ is \emph{supported at }$0$. Thus, $\widehat{h-f} = \sum_{k=0}^N a_k \partial_x^k$ for some $N \in \N$. Since this implies that $h-f$ is a polynomial, and since $\langle h-f,g'\rangle = 0$ for all $g \in \mathcal{S}(\R)$, we conclude that $h-f$ is a constant. 

In order to conclude the proof, note that $f_0(x) = \int_0^x f'(t) \, dt$ satisfies that $f_0' = f'$ in the sense of distributions. Since $f_0$ is (H\"older) continuous, and since by the considerations above $f_0$ and $f$ differ by a constant, $f$ is continuous, as desired. 
\end{proof} 

\subsection{Main equivalences and the operator $T$}\label{sec:equivalences}
In this section, we shall begin to analyze the operator described in the Introduction, which maps a ``reasonable'' function $f$ to the restriction to the $y$-axis of the solution $u$ to the Klein-Gordon equation 
\[
\partial_x \partial_y u + \pi^2 u = 0,
\]
with $u(x,0) = f(x).$ Indeed, we shall define the operator $T$ by its action on functions $\theta \in C^{\infty}_c(\R)$ such that $\theta(0) = 0$ as 
\begin{equation}\label{eq:final-def-T-0}
T\theta(\xi) = -\pi \sqrt{|\xi|} \int_{\{y \colon \xi y < 0\}} \frac{\theta(y)}{\sqrt{|y|}} J_1(2\pi \sqrt{|\xi y|}) \, dy,
\end{equation}
where $J_{\nu}$ denotes the Bessel function of the first kind and order $\nu \in \R.$ Clearly, for $\theta \in C^{\infty}_c(\R),$ the operator $T$ above is well-defined and finite for each $\xi \in \R.$ 

The definition of this operator seems cumbersome and unrelated to the Klein-Gordon equation at first, but, as we shall see below, a reformulation of it shows that $T$ is the most ``natural'' operator to consider when analyzing Heisenberg uniqueness pairs involving the hyperbola $\Gamma$, and in fact matches the previously stated relationship. For that purpose, we start with the following result.

\begin{lemma}\label{lemma:radial-T} For $\theta \in C^{\infty}_c(\R)$ with $\theta(0)=0$, define $\Theta_{\pm}\colon\R^4\to\C$ by
\begin{equation}\label{eq:def-big-Theta} 
\Theta_{\pm}(u) = \frac{\theta(\pm|u|^2)}{|u|^2}.
\end{equation}
Then
\begin{equation}\label{eq:T-Fourier-dim-4}
T\theta(\xi) = - |\xi| \mathcal{F}_4(\Theta_{-\sgn(\xi)})(\sqrt{|\xi|}),
\end{equation}
where $\mathcal{F}_4$ denotes the Fourier transform in dimension 4. 
\end{lemma}

\begin{proof} We suppose that $\xi < 0,$ with the other case being entirely analogous. We recall the formula for the Fourier transform of a radial function $g:\R^4 \to \C$ in $L^1(\R^4)$: 
\[
\mathcal{F}_4(g)(\eta) = 2\pi \int_0^{\infty} g(s) \left( \frac{s}{|\eta|}\right) J_1(2 \pi |\eta| s) \, s \, ds.
\]
Changing variables $y = s^2$ in the definition \eqref{eq:final-def-T-0}, we get 
\begin{align*} 
T\theta (\xi)  &= -2\pi \sqrt{|\xi|} \int_0^{\infty} \theta(s^2) J_1(2\pi \sqrt{|\xi|} s) \, ds \cr 
        &= -2\pi |\xi| \int_0^{\infty} \frac{\theta(s^2)}{s^2} \left(\frac{s}{\sqrt{|\xi|}}\right) J_1(2\pi\sqrt{|\xi|}s) \, s \, ds,
\end{align*} 
which concludes the proof. 
\end{proof}

We are now able to prove our main equivalence for the definition of $T$, which enables us to connect it to Heisenberg Uniqueness pairs: 

\begin{proposition}\label{prop:def-T} Under the conditions above, if $\psi \in L^1$ is such that $\widehat{\psi} = \theta$, then 
\begin{equation}\label{eq:T-fourier-invert}
    T\theta(\xi) = \frac{1}{\sqrt{2}}\int_{\R} \psi(t) e^{-\pi i \xi/t} \, dt. 
\end{equation}
\end{proposition}

\begin{proof} Let us suppose that $\xi < 0$, since the other case is entirely analogous. Then the considerations in \cite[Section~2.1]{Bakan-Hedenmalm-Montes-Radchenko-Viazovska} 
 together with Lemma \ref{lemma:radial-T} imply the following: if 
 $$\theta(t) = \frac{1}{\sqrt{2}} \int_{\R} e^{-\pi i t s} \psi(s) \, ds,$$  
 then the four-dimensional Fourier transform of $\Theta$ at $\eta$ is given by 
\[
-\frac{1}{\sqrt{2}|\eta|^2} \int_{\R} e^{\pi i |\eta|^2/t} \psi(t) \, dt. 
\]
Here, we made crucial use of the fact that $\theta(0) = \widehat{\psi}(0) = 0$, in order to use the results from \cite{Bakan-Hedenmalm-Montes-Radchenko-Viazovska}. Replacing then $\eta$ by $\sqrt{|\xi|}$ and using that $|\xi| = -\xi$ yields the claim. 
\end{proof}

\begin{remark} Proposition \ref{prop:def-T} is similar in spirit to Proposition 5.2.1 in \cite{Hedenmalm-Monte-Klein-Gordon}: indeed, by simply using a formal Fubini argument in \eqref{eq:T-fourier-invert} with that result, one may (formally) recover Proposition \ref{prop:def-T}. As a matter of fact, the formal argument outlined here may be made formal through multiplication with a decay-inducing factor and then taking limits. For brevity of exposition, however, we decided not to include such computations. 

A crucial feature of Proposition \ref{prop:def-T} is that it allows us to extend the definition of $T$ to general continuous, bounded functions $\theta$, satisfying $\theta = \widehat{\psi}$, where $\psi \in L^1$ and $\widehat{\psi}(0) = 0$. Indeed, we may always write $T\theta = \widehat{\varphi}$ whenever $\theta \in C^{\infty}_c(\R)$, where $\varphi(t) = t^{-2} \psi(1/t)$. Since $\varphi \in L^1$ whenever $\psi \in L^1$, we may use this equivalent definition to extend $T$ to the more general class of $\theta \in \big(\widehat{L^1}\big)_0 = \{\theta = \widehat{\psi},\psi \in L^1, \theta(0)=0\}.$

Finally, note that if $u$ is a solution to $\partial_x \partial_y u + \pi^2 u =0$, with the additional constraint that $u(x,0) = \theta(x)$, then if we may write $\widehat{\psi} = \theta$, it follows from a direct Fourier characterization argument that $u(x,y) =\frac{1}{\sqrt{2}} \int_{\R} \psi(t) e^{-\pi i (x \cdot t + y/t)} \, dt$, and Proposition \ref{prop:def-T} shows that $u(0,y) = T\theta(y)$, concluding hence the claim about the relationship between $T$ and the Klein-Gordon equation.  
\end{remark} 

In the next result we highlight some other important quantitative facts about the operator~$T$ in terms of decay and regularity: 

\begin{proposition}\label{prop:properties-T} Under the same hypotheses of Proposition \ref{prop:def-T}, the following assertions hold: 
\begin{enumerate}[\normalfont(I)]
    \item For all admissible $\theta$, we have that 
    \begin{align*}
    \|T\theta\|_{L^2(dt/|t|)} =  \,\, &\|\theta\|_{L^2(dt/|t|)}, \cr 
    \|T\theta\|_{L^2(\R)} = \frac{1}{\sqrt{2}\pi} &\|\theta'\|_{L^2(\R)}. 
    \end{align*}
    \item $(T\circ T)\theta = \theta$ holds for any $\theta$ admissible;
    \item $\textup{supp}(\theta) \subset \R_+$ if and only if $\textup{supp}(T\theta) \subset \R_-$. 
\end{enumerate}
\end{proposition}

\begin{proof} We start with the proof of the first equality in (I). We have, by \eqref{eq:def-big-Theta}, \eqref{eq:T-Fourier-dim-4} and Plancherel, 
\[
\int_{\R_-} \frac{|T\theta(\xi)|^2}{|\xi|} \, d\xi =  2\int_{\R_+} |\mathcal{F}_4(\Theta)(t)|^2 t^3 \, dt =  2\int_{\R_+} |\Theta(t)|^2 t^3 \, dt = \int_{\R_+} \frac{|\theta(s)|^2}{|s|} \, ds.
\]
By doing the same for the integral over $\R_+$, we conclude that part. For the second equality, we may write, with the aid of Proposition \ref{prop:def-T}, 
\[
\|T\theta\|_2^2 = \int_{\R} \left|t^{-2}\psi(1/t)\right|^2 \, dt =  \int_{\R} t^2 |\psi(t)|^2 \, dt = \frac{1}{2\pi^2} \|\theta'\|_2^2,
\]
as desired. For (II), note that $T\theta$ is admissible if $\theta$ is: indeed, Part (I) shows that $T\theta \in L^2(dt/|t|)$. Moreover, since $T$ is given by the Fourier transform of the function $t^{-2} \psi(1/t) \in L^1$, it is continuous. Finally, an analogous argument to the one in Part (I) with the definition of $T$ shows that $(T\theta)' \in L^2$. Hence, the operator $(T\circ T)\theta$ is well-defined. 

For the desired identity, we first note that 
\[
\widecheck{T\theta}(y) =  y^{-2} \psi(1/y),
\]
and thus 
\[
\theta(t) = \widehat{\psi}(t) =  \frac{1}{\sqrt{2}} \int_{\R} y^{-2} \widecheck{T\theta}(1/y) e^{-\pi i y t} \, dy = \frac{1}{\sqrt{2}} \int_{\R} \widecheck{T\theta}(y) e^{-\pi i t/y}\, dy =  (T \circ T)\theta(t). 
\]
Finally, Part (III) follows directly from \eqref{eq:final-def-T-0}. 
\end{proof}

Note the following direct consequence of Proposition \ref{prop:properties-T}: $\theta \in L^2(\R) \cap H^1(\R)$ if, and only if, $T\theta \in L^2(\R) \cap H^1(\R)$. We use this fact to state the following main equivalence result, which will be crucial throughout the rest of the paper: 

\begin{proposition} The following assertions are equivalent:
\vspace{2mm}

\begin{enumerate}[\normalfont(I)]
    \item There exists $\psi \in L^2((1+t^2)dt)$ with $\widehat{\psi}(0) = 0$ and $\int_{\R} \psi(t) \, e^{-\pi i a_n t} \, dt = 0 = \int_{\R} \psi(t) \, e^{- \pi i b_n/t} \, dt = 0, \, \forall \, n \in \Z$;
    \item There exists $\theta \in L^2(dt/|t|)\cap L^2(\R) \cap H^1(\R)$, such that $\theta(a_n) = T\theta(b_n) = 0, \, \forall n \in \Z$.
\end{enumerate}
\end{proposition}

The proof of this result is a direct consequence of Proposition \ref{prop:embedding}, Proposition \ref{prop:def-T} and Proposition \ref{prop:properties-T}, and its statement will provide us with the right framework for both our positive and negative results. 

\section{Proof of Theorem \ref{thm:negative-gen}}\label{sec:proof-neg}

\subsection{Preliminary Reductions} In order to begin our process of building a function $\psi$ satisfying the hypotheses of Theorem \ref{thm:negative-gen}, we start by recalling some concepts and reductions from \cite{Kulikov-Nazarov-Sodin}. 

Indeed, first and foremost, we define the \emph{density of order $p$} (if it exists) of a sequence $Z = \{z_j\}_{j\in\Z}$ of points in the complex plane as 
$$\mathcal{D}_p(Z) = \lim_{r \to \infty} \frac{\#(Z\cap D_r(0))}{r^p}.$$ 
We also need the concept of \emph{smooth sequences}, as defined in \cite{Kulikov-Nazarov-Sodin}: a sequence $Z$ with $p$-density $\mathcal{D}_p(Z)$ is called \emph{$p$-smooth} if $z_j = r_j e^{i\theta}$, for some fixed $\theta \in [0,2\pi)$, and the following conditions are met: 
\begin{enumerate}[(I)]
    \item The function $r \mapsto \left| \#(Z \cap D_r(0)) - \mathcal{D}_p(Z) r^p\right|$ is bounded for $r>0$;
    \item There is $d>0$ such that $r_{j+1} - r_j \ge d(1+r_j)^{1-p}$. 
\end{enumerate}

Our first reduction will show that, if $A$ and $B$ are sequences as in the statement of Theorem \ref{thm:negative-gen}, then we may `expand' them a bit so that they become more regular, while still satisfying the same asymptotic properties in terms of density: 

\begin{lemma}\label{lemma:p-smooth-seq} Let $C = \{c_n\}_{n \ge 0}$ be an increasing sequence in $\R_+$, such that, for some $\sigma > 0$, 
\[
\liminf_{n \to \infty} (c_{n+1} - c_n) > \sigma .
\]
Then, for any $D > 1/\sigma$, there exists a $1$-smooth sequence $C' \supset C$ with $\mathcal{D}_1(C') = D$ and such that $|C'\setminus C| = +\infty$.
\end{lemma}

We refer the reader to \cite[Claim~7]{Kulikov-Nazarov-Sodin} for a proof of this fact. Using Lemma \ref{lemma:p-smooth-seq}, we then replace $A$ and $B$ by $A'$ and $B'$, where $A' \cap (0,+\infty)$ is a sequence given by Lemma \ref{lemma:p-smooth-seq} for $A \cap (0,+\infty)$, with its definition for negative values being analogous. The set $B'$ is built in an entirely analogous way from $B$.   By the Lemma again, we may still assume that $\mathcal{D}_1(A'), \mathcal{D}_1(B') < 1$. We shall then work with $A'$, $B'$ throughout the steps of the proof below.

\subsection{Entire functions with prescribed zeros} Let $k:[-\pi,\pi] \to \R$ be defined by 
\begin{equation}\label{eq:k-def-2}
k(\theta) = (\pi \beta) \sin(2 \theta) - \gamma \cos(2 \theta)
\end{equation} whenever $\theta \in [0,\pi/2],$ and extended to $\theta \in [-\pi,\pi]$ such that it becomes even and symmetric with respect to $\pi/2.$ For such $k,$ we call a discrete set $\mathcal{Z} = \mathcal{Z}_1^+ \cup \mathcal{Z}_1^{-} \cup \mathcal{Z}_2^+ \cup \mathcal{Z}_2^{-} \subset \C,$ where $\mathcal{Z}_j^{\pm} \subset \left\{ \arg (z) = \frac{(j\pm1) \pi}{2} \right\},$ \emph{$k-$regular} if:

\begin{enumerate}[(I)]
\item $\mathcal{Z}_j^{\pm}$ has density $m_j/4\pi$ with respect to exponent $2.$ That is, it holds that 
$$|\mathcal{Z}_j^{\pm} \cap (0,r e^{i\theta_j^{\pm}})| \sim r^2m_j/ 4\pi$$
as $r \to \infty$, where $m_1 = 2 \pi \beta, m_2 = 2\gamma,$ with $\theta_j^{\pm}=\frac{(j\pm1) \pi}{2}$;

\vspace{2mm}

\item the disks $\{D_z\}_{z \in \mathcal{Z}} = \{B(z,c\cdot (1+|z|)^{-1})\}_{z \in \mathcal{Z}}$ are all \emph{disjoint} for some $c>0.$  
\end{enumerate} 
\vspace{2mm}

We now state the following result by Levin, which will be crucial in our construction: 

\begin{theorem}[Theorem~5, Ch. II, \cite{Levin}]\label{thm:levin}
Let \( k \) be as in \eqref{eq:k-def-2}. Then, for any \( k \)-regular set \( \mathcal{Z} \), there exists an entire function \( S \), whose zeroes are simple and coincide with \( \mathcal{Z} \), such that, for every \( \varepsilon > 0 \),
\begin{align}
    |S(re^{i\theta})| \le C_{\varepsilon} e^{(k(\theta) + \varepsilon)r^2},  & \text{ everywhere in } \mathbb{C}, \tag{I} \\
    |S(re^{i\theta})| \ge c_{\varepsilon} e^{(k(\theta) - \varepsilon)r^2}, & \text{ whenever } w = re^{i\theta} \not\in \cup_{z \in \mathcal{Z}} D_z. \tag{II}
\end{align}
\end{theorem}

We refer the reader to either the original work of Levin \cite{Levin} or the recent manuscript \cite{Kulikov-Nazarov-Sodin}, where a self-contained proof of this fact, which works in our context, is available in the appendix. 

As consequence of (I) and (II) in Theorem~\ref{thm:levin},  we have that, for $z \in \mathcal{Z}_j,$ 
\begin{equation}\label{eq:bound-derivative-S}
|S'(z)| \ge c_{\varepsilon} e^{(k(\arg z) - \varepsilon)|z|^2}.
\end{equation}
This can be seen, for instance, by applying the mean value property to the harmonic function $\log(|S(w)/(w-z)|)$ in the disk $D_z.$

\subsection{Almost-interpolation bases} Our next step is to construct functions by using Theorem \ref{thm:levin} which will give rise to our iteration scheme. We relabel the sequences $A$ and $B$ so that $$A \cap (0,+\infty) = \{a_k\}_{k \ge 1},\qquad B \cap (0,+\infty) = \{b_k\}_{k \ge 1}. $$
 In what follows, we will mainly deal with the set $A$, but the same construction can be repeated to~$B$ by a verbatim adaptation. 
 
 We define $\tilde{A}_1$ to be the sequence of elements of the form $\sqrt{a_n}$ whenever $n > 0$, and we let $\tilde{A}_2$ analogously be the sequence of elements of the form $\sqrt{-a_n}$ whenever $n \le 0$. 

We start by taking $\mathcal{Z}_1^+(1) = \tilde{A}_1,$ and setting $\mathcal{Z}_1^{-}(1) = -\mathcal{Z}_1^{+}(1).$ We then add two other sets $\mathcal{Z}_2^{\pm}(1)$ so that the properties in the definition of $k$-regularity are fulfilled, and $-\mathcal{Z}_2^{+}(1) = \mathcal{Z}_2^{-}(1).$ The exact choice of these sets is not important, as long as the set $\mathcal{Z}(1) := \cup_{\pm} \cup_{j=1,2} \mathcal{Z}_j^{\pm}(1)$ obtained in the end is origin-symmetric and is $k_1$-regular, for some function $k_1$ as above. By repeating the same construction for $\mathcal{Z}_1^+(2) = \tilde{A}_2$, we obtain in a similar manner a $k_2$-regular set $\mathcal{Z}(2)$ associated to a (possibly) different function $k_2$. Note, however, that the functions $k_1, k_2$ are (crucially) only tied to $\mathcal{Z}(1), \, \mathcal{Z}(2)$ through the densities of these sets. 

We then use Theorem \ref{thm:levin} for $\mathcal{Z}(i), i \in \{1,2\}.$ This yields an entire function $S_i$ in each case with the desired zero properties at the sets $\mathcal{Z}(i)$, respectively. Moreover, we may assume that such a function is always \emph{even}. This follows directly from the construction of $S_i$, as given either in \cite{Levin} or \cite{Kulikov-Nazarov-Sodin}: if we let $E_2(z) = (1-z)e^{z + z^2/2}$, then $S_i$ is obtained as a limit as $R \to \infty$ of the functions 
\[
S_{i,R}(z) = e^{\alpha z + \beta z^2} \prod_{\lambda \in \mathcal{Z}_i \cap D_R(0)} E_2(z/\lambda),
\]
where $\alpha,\beta \in \C.$ We claim that, if $S_i$ satisfies the conditions of Theorem \ref{thm:levin}, then so does $R_i(z) := e^{- \alpha z} \cdot S_i(z).$ Indeed, the zero assumption does not change by multiplying by complex exponentials. Moreover, since $|e^{\alpha z}| \le C_{\varepsilon} e^{\varepsilon |z|^2}$ for all $\varepsilon > 0$, the bounds (I) and (II) in Theorem \ref{thm:levin} are fulfilled. Therefore, we may suppose without loss of generality that $\alpha = 0$ in the definition of $S_{i,R}$. 

Since we have shown that we may take $S_{i,R}$ to be even by construction for each $R$, then we may also take $S_i$ to be even, as desired. Consider then, for a fixed $a \in \tilde{A}_i$, the function 
\[
\mathcal{S}_{i,a}(z) = \frac{S_i(z)}{z^2 - a^2}. 
\]
Since $S_i$ is even, this new function is again an entire function. By applying the maximum modulus principle to this function on a small disk around $a,$ and using the bounds we have on $S_i$ outside it, we can easily conclude that it satisfies 
\begin{equation}\label{eq:S-bound}
|\mathcal{S}_{i,a}(z)| \lesssim C_{\varepsilon} e^{(k(\theta)+\varepsilon)|z|^2}
\end{equation}
in all of $\C$, uniformly on $a \in \tilde{A}_i$. 

Now, for any $n \in \Z$, let $a(n)$ be the positive real number for which $a(n)^2 = \pm a_n$, $a(n) \in \tilde{A}_1 \cup \tilde{A}_2$. We then define a function $\theta_n:\R \to \C$ in the following way: if $a(n) \in \tilde{A}_1$, we define 
\[
\theta_n(t) = \begin{cases}
               t \cdot \mathcal{S}_{1,a(n)}(\sqrt{t}), & \text{ if } t > 0, \cr 
               0, & \text{ if } t < 0. 
             \end{cases} 
\]
The definition of $\theta_n$ for $a(n) \in \tilde{A}_2$ is analogous. These functions will serve as ``approximate'' interpolation functions: $\theta_{n}$ vanishes on $A$, except for the point $a_n$. Although no analogous interpolation property holds for $T(\theta_{n})$, we have a (weak) surrogate of that through the following strong decay bounds:

\begin{lemma}\label{lemma:preliminary-2} There exist absolute constants $\alpha' > \alpha'' > 0$ such that the following holds. For each $x \in \R$, each $\xi \in \R$ and each $n \in \Z$, we have   
\[
|\theta_{n}(x)| \le C e^{-\alpha''|x|}, \, \, \, \, 
|T(\theta_{n})(\xi)| \le C e^{-\alpha'|\xi|}. 
\] 
\end{lemma}
\begin{proof} Without loss of generality, let us assume that $a_n > 0.$ By Proposition \ref{prop:properties-T}, we have that $\text{supp}(T\theta_n) \subset \R_-,$ and hence it is enough to prove the bound for $\xi < 0.$ We shall write $a:= a(n)$ throughout the proof. 

We then recall \eqref{eq:T-Fourier-dim-4}, which, for $\Theta_{n}(u) = \frac{\theta_{n}(|u|^2)}{|u|^2}$ for $u \in \R^4$, allows us to write 
\begin{equation}\label{eq:T-fourier}
T(\theta_{n})(\xi) = - |\xi| \cdot \mathcal{F}_4(\Theta_{n})(\sqrt{|\xi|}).
\end{equation}
By recalling the definition of $\theta_{n}$, we have that 
\[
\mathcal{F}_4(\Theta_n)(\eta) = \int_{\R^4} \mathcal{S}_{1,a}(|x|) e^{-2 \pi i x \cdot \eta} \, d x.  
\]
On the other hand, $\mathcal{S}_{1,a}$ is an \emph{even} function of $z$, which allows it to be written as $\mathcal{S}_{1,a}(z) = H_{1,a}(z^2),$ for some entire function $H_{1,a}$. Note that $|H_{1,a}(re^{i\theta})|\le e^{-cr}$ if $\theta$ is close to either $0$ or $\pi$, by the properties of $\mathcal{S}_{i,a}$.  

Hence, taking these considerations into account, we may write the integral defining $\mathcal{F}_4(\Theta_n)$ through Fubini's theorem as a four-fold integral, as follows: 
\begin{align*}
\mathcal{F}_4(\Theta_n)(\eta) = \int_{\R^4} H_{1,a}(x_1^2 + x_2^2 + x_3^2 + x_4^2) \, e^{-2\pi i x \cdot \eta} \, dx. 
\end{align*}
By the previously mentioned decay properties of $H_{1,a},$ we are able to change contours in each one of those one-dimensional integrals. This allows us to write
\begin{equation}\label{eq:fourier-4}
\mathcal{F}_4(\Theta_n)(\eta) = \int_{\R^4} H_{1,a}(\|x+iy\|^2)e^{-2 \pi i (x+iy) \cdot \eta } \, dx.
\end{equation}
Here, we use the convention that, for $x,y \in \R^4$, $\|x+iy\|^2 := \sum_{j=1}^4 (x_j + i y_j)^2.$ From \eqref{eq:fourier-4}, we may bound $|\mathcal{F}_4(\Theta_n)(\eta)|$ as follows. First, let $y = -t \eta,$ where $t>0$ will be chosen later. Then 
\begin{align} 
|\mathcal{F}_4(\Theta_n)(\eta)| & \le e^{-2\pi |y||\eta|} \left| \int_{\R^4} |H_{1,a}(\|x+iy\|^2)| e^{-\varepsilon r} e^{\varepsilon r}\, d x \right| \cr 
 \label{eq:first-bound-H-lambda-1} & \le e^{-2 \pi |y||\eta|} \sup_{x \in \R^4} \left( | H_{1,a}(\|x+iy\|^2) e^{\varepsilon r} | \times \left( \int_{\R^4} e^{-\varepsilon r} \, dx \right) \right). 
\end{align} 
where we write $\|x+iy\|^2 = |x|^2 - |y|^2 + 2 i \langle x, y \rangle =: re^{i\theta}.$ Here, $\varepsilon>0$ will be chosen later. We now note that 
$$r = \left( (|x|^2 - |y|^2)^2 + 4 \langle x,y \rangle^2 \right)^{1/2} \ge ||x|^2 - |y|^2|, $$
which implies 
\[
\int_{\R^4} e^{-\varepsilon r} \, dx \le C \int_0^{\infty} e^{-\varepsilon|s^2 - |y|^2|} \, s^3 \, ds = \frac{C}{2} \int_0^{\infty} e^{-\varepsilon|s' - |y|^2|} \, s' \, ds'.
\]
\vspace{1mm}
On the other hand, it is not hard to see that the expression on the right-hand side above is bounded by $C_\varepsilon (1+|y|^2)$. Thus, from \eqref{eq:first-bound-H-lambda-1}, we obtain 
\[
|\mathcal{F}_4(\Theta_n)(\eta)| \le C_{\varepsilon}(1+|y|^2) e^{-2\pi |y||\eta|} \sup_{x \in \R^4} \left( | H_{1,a}(\|x+iy\|^2)| e^{\varepsilon r} \right).
\]
From the bounds available on $\mathcal{S}_{1,a}$, we conclude that $|H_{1,a}(re^{i\theta})| \le C_{\varepsilon} e^{(k(\theta/2) + \varepsilon)r}$. Thus, 
$$|H_{1,a}(\|x+iy\|^2)| e^{\varepsilon r} \le C_{\varepsilon} e^{(k(\theta/2) + 2\varepsilon)r}. $$ 
We now note that, if $s<1$ is sufficiently close to 1, and $\gamma = \pi s \omega^{-2},$ with $\omega$ sufficiently large, then any function $k$ defined as in \eqref{eq:k-def-2} satisfies
\begin{equation}\label{eq:bound-on-k}  
k(\theta) < \pi \omega^2 \sin^2(\theta), \, \text{ whenever } \theta \in [0,\pi/2]. 
\end{equation} 
For a proof of this fact, we refer the reader to \cite[Claim~8]{Kulikov-Nazarov-Sodin}. Since the inequality in \eqref{eq:bound-on-k} is \emph{strict} for the whole compact interval $[0,\pi/2],$ if we choose $\varepsilon > 0$ small enough, we will have $k(\theta/2) + 2\varepsilon \le \pi \omega^2 \sin^2(\theta/2),$ whenever $\theta \in [-\pi,\pi].$ Thus, for such $\varepsilon > 0$, we obtain the following bound:
\[
|\mathcal{F}_4(\Theta_n)(\eta)| \le C' (1+|y|^2) \sup_{x \in \R^4} e^{2 \pi \left(\frac{\omega^2}{2} \sin^2(\theta/2) r - |y||\eta|\right)}. 
\]
On the other hand, $r \le |x|^2 + |y|^2$ and $\sin^2(\theta/2) = \frac{1}{2}(1-\cos(\theta)),$ which implies that $\sin^2(\theta/2) r = \frac{1}{2} (r - r \cos(\theta)) \le \frac{1}{2}(|x|^2 + |y|^2 - (|x|^2 - |y|^2)) = |y|^2.$ This shows finally that 
\begin{equation}\label{eq:almost-final-bound-G-lambda}
|\mathcal{F}_4(\Theta_n)(\eta)| \le C'(1+|y|^2) \sup_{x \in \R^4} e^{2\pi \left( \frac{\omega^2}{2} |y|^2 - |y||\eta|\right)}.
\end{equation}
Recalling that $y = -t \eta, \, t>0,$ and optimizing on the parameter~$t$, we get from \eqref{eq:almost-final-bound-G-lambda} that 
\begin{equation}\label{eq:final-bound-4-dim-fourier} 
|\mathcal{F}_4(\Theta_n)(\eta)| \lesssim (1+|\eta|^2) e^{-\frac{\pi}{\omega^2}|\eta|^2} = (1+|\eta|^2) e^{-\frac{\gamma}{s} |\eta|^2}.
\end{equation} 
Gathering \eqref{eq:final-bound-4-dim-fourier} with \eqref{eq:T-fourier}, the definition of $\theta_n$, \eqref{eq:S-bound} and \eqref{eq:k-def-2}, one is able to finish the proof. 
\end{proof}

With Lemma \ref{lemma:preliminary-2} at hand, we define
$$\varrho_n(t) := \frac{\theta_n(t)}{a(n)^2 \cdot \mathcal{S}_{i,a(n)}(a(n))}.$$
By \eqref{eq:bound-derivative-S}, it follows that $|\mathcal{S}_{i,a(n)}(a(n))| \ge c_{\varepsilon} e^{-(\gamma+\varepsilon)|a_n|^2}$ holds for each $n$, which implies, together with Lemma \ref{lemma:preliminary-2}, that there are $\alpha' > \alpha > \alpha''>0$ such that 
\begin{align*}
    |\varrho_n(t)| &\le C e^{-\alpha''|t| + \alpha|a_n|}, \cr 
    |T(\varrho_n)(\xi)| &\le C e^{-\alpha'|\xi| + \alpha|a_n|}.
\end{align*}
Moreover, we have that $\text{supp}(\varrho_n) \subset \R_{\sgn(n)}$, where we use the convention that $\sgn(0) = -1$. Finally, note that $\varrho_n(a_j) = \delta_{n,j}$ whenever $\sgn(n) = \sgn(j)$, and that, since $T$ is an isometry between $H^1$ and $L^2$, we have that 
\[
\|\varrho_n'\|_2 = c \cdot \|T(\varrho_n)\|_2 \le C e^{\alpha|a_n|}.
\]
By running the same argument in the exact same fashion as we did above, we are also able to find a sequence of functions $\sigma_n$ with the property that 
\begin{align*}
    |\sigma_n(t)| &\le C e^{-\alpha''|t| + \alpha|b_n|}, \cr 
    |T(\sigma_n)(\xi)| &\le C e^{-\alpha'|\xi| + \alpha|b_n|},
\end{align*}
with $\text{supp}(\sigma_n) \subset \R_{\sgn(n)}$, $\sigma_n(b_j) = \delta_{n,j}$ whenever $\sgn(n) = \sgn(j)$, and $\|\sigma_n'\|_2 \le C e^{\alpha|b_n|}.$ 

\vspace{2mm}

\subsection{Iteration Scheme and Conclusion}\label{sec:iteration-funct} We are now ready to set up the main iteration process. Fix $\delta > 0$ sufficiently small (to be chosen later), and take $L > 0$ such that 
\[
\sum_{k \in A_L} e^{(\alpha - \alpha')|a_k|} + \sum_{j \in B_L} e^{(\alpha-\alpha')|b_j|} < \delta,
\]
where $A_L = \{ k \in \Z \colon |a_k| > L\}, \, B_L = \{ j \in \Z \colon |b_j| > L\}.$ We further divide these sets into 
\[
A_L^+ = A_L \cap (0,+\infty), \,A_L^{-} = A_L \setminus A_L^+,\, B_L^+ = B_L \cap (0,+\infty), \, B_L^- = B_L \setminus B_L^+. 
\]
With that in mind, we define the Banach spaces $\mathfrak{B}^{\pm}$ of pairs of sequences $\kappa = (\{s_k\}_{k \in A_L^{\pm}},\{r_j\}_{j \in B_L^{\mp}})$ such that the norm 
\begin{equation}\label{eq:finite-norm}
\|\kappa\|_{\mathfrak{B}^{\pm}} := \sum_{k \in A_L^{\pm}} |s_k| e^{\alpha|a_k|} + \sum_{j \in B_L^{\mp}} |r_j|e^{\alpha|b_j|} < +\infty. 
\end{equation}
We are then able to show the following: 
\begin{proposition}\label{prop:interpolation} For each sequence $\kappa \in \mathfrak{B}^{\pm}$, there exists a continuous function $f_{\pm}:\R \to \C$ with $f_{\pm} \in L^2 \cap H^1$ such that $\supp(f_{\pm}) \subset \R_{\pm}$ and 
\begin{align*}
    f_{\pm}(a_k) & = s_k, \, \forall k \in A_L^{\pm}, \cr 
    Tf_{\pm}(b_j) &= r_j, \, \forall j \in B_L^{\mp}. 
\end{align*}
\end{proposition}

\begin{proof} Suppose without loss of generality that we are dealing with $\mathfrak{B}^+$. We start by taking 
\begin{equation}\label{eq:first-interpol-seq}
f_0(t) = \sum_{k \in A_L^+} s_k \varrho_k(t) + \sum_{j \in B_L^-} r_j T(\sigma_j)(t). 
\end{equation}
Since $\supp(\sigma_j) \subset \R_-$ if $j<0$, we have that $\supp(f_0) \subset \R_+$. Moreover, it follows directly from the definition that we have the following pointwise bound on $f_0$ and $T(f_0)$:
\begin{equation}\label{eq:bound-f-Tf}
|T(f_0)(t)|+|f_0(t)| \le C \|\kappa\|_{\mathfrak{B}^+} e^{-\alpha''|t|}. 
\end{equation}
Furthermore, we may bound the $H^1$ norm of $f_0$ as follows: 
\begin{align}\label{eq:L^2-f_0}
\|f_0'\|_2 &\le \sum_{k \in A_L^+} |s_k| \|\varrho_k'\|_2 + \sum_{j \in B_L^{-}} |r_j|\|T(\sigma_j)'\|_2 \cr 
            &\le C \left(\sum_{k \in A_L^+} |s_k| e^{\alpha|a_k|} + \sum_{j \in B_L^-} |r_j| e^{\alpha|r_j|}\right) \le C\|\kappa\|_{\mathfrak{B}^+}. 
\end{align}
In a completely analogous manner, by arguing with $(Tf_0)'$ instead and using that $T(Tf_0) = f_0$, we are able to conclude that $\|f_0\|_2 \le C \|\kappa\|_{\mathfrak{B}^+}$ as well.

We now define a linear map $\mathcal{T}$ from $\mathfrak{B}^+$ to itself, by $\mathcal{T}\kappa = (\{(\mathcal{T}^1\kappa)_k\}_{k \in A_L^+}, \{(\mathcal{T}^2 \kappa)_j\}_{j \in B_L^{-}})$, where 
\begin{align*}
    (\mathcal{T}^1\kappa)_k &= f_0(a_k) - s_k, \cr 
    (\mathcal{T}^2\kappa)_j &= T(f_0)(b_j) - r_j. 
\end{align*}
Now, we simply notice that 
\[
\sum_{k \in A_L^+} |(\mathcal{T}^1\kappa)_k| e^{\alpha |a_k|} \le C \sum_{k \in A_L^+} \sum_{j \in B_L^{-}} |r_j| e^{(\alpha-\alpha')|a_k| + \alpha|b_j|}
\]
\[
\le C \|\kappa\|_{\mathfrak{B}^+} \sum_{ k \in A_L^+} e^{(\alpha-\alpha')|a_k|} < C \delta \|\kappa\|_{\mathfrak{B}^+}. 
\]
The same argument applied to $(\mathcal{T}^2\kappa)_j$ shows that 
\[
\|\mathcal{T} \kappa\|_{\mathfrak{B}^+} \le 2C\delta \|\kappa\|_{\mathfrak{B}^+}. 
\]
Take thus $\delta < \frac{1}{4C}$ above. We now define $f_1$ analogously as in \eqref{eq:first-interpol-seq}, but with $\mathcal{T}\kappa$ in place of $\kappa$, and in general we define $f_n$ recursively to be the function interpolating $\mathcal{T}^{(n)}\kappa$ as coefficients, as in \eqref{eq:first-interpol-seq}. Here, $\mathcal{T}^{(n)}$ denotes the $n-$th iterate of the map $\mathcal{T}$ defined above. We claim that 
\[
f_+ := -\sum_{n \ge 0} f_n
\]
satisfies the hypotheses of the proposition. Indeed, since we showed that $\mathcal{T}$ is a contraction in $\mathfrak{B}^+$, \eqref{eq:bound-f-Tf} implies that the function $f_+$, as well as $Tf_+$, is pointwise bounded by $e^{-\alpha''|t|}$. The telescoping nature of the definition of $\mathcal{T}$ then implies that $f_+(a_k) = s_k$ for each $k \in A_L^+$, as well as $Tf_+(b_j) = r_j,$ for all $j \in B_L^{-}$. Moreover, since $f_+$ is an uniform limit of continuous functions, it is itself continuous, and since $\supp(f_n) \subset \R_+$, the same holds for $f_+$. Finally, from \eqref{eq:L^2-f_0} we obtain that $f_+\in L^2 \cap H^1$ as desired, with the bound 
\[
\|f_+\|_2 + \|f_+'\|_2 \le C \|\kappa\|_{\mathfrak{B}^+}
\]
being additionally fulfilled. 
\end{proof}

\begin{proof}[\unskip\nopunct]  Recall that, in the beginning of the proof, we replaced $A$ and $B$ by $A'$ and $B'$, with $|A' \setminus A| = |B' \setminus B| = +\infty.$ It follows thus that the arguments above for such modified sets results, in the same way, in a solution to the respective interpolation problem. We let then $s_k,r_j$ be sequences such that $s_k = 0 = b_j$ if $k$ and $j$ are indices of an element of the original sets $A$ and $B$, respectively, and set them as arbitrary values in the complement $A' \setminus A, \, B' \setminus B$, with the only constraint that they do not violate \eqref{eq:finite-norm}. This shows that the space of functions which vanish on $A_{\pm} \setminus [-L,L]$ with their image under $T$ vanishing on $B_{\mp} \setminus[-L,L]$ is infinite-dimensional.

On the other hand, since there are only finitely many points of $A_{\pm}$ and $B_{\mp}$ inside $[-L,L]$, we conclude that there must be an infinite-dimensional space of functions $F_{\pm}$ such that 
\[    F_{\pm}(a_k) = 0, \, \forall k \text{ such that } \sgn(k) = \pm, 
\]
\[
    TF_{\pm}(b_j) =0, \, \forall j \text{ such that } \sgn(j) = \mp. 
\]
Since the support of $F_{\pm}$ is in $\R_{\pm}$, and the support of $TF_{\pm}$ is in $\R_{\mp}$, the vanishing statement \eqref{eq:HUP-cond-gen}  $\psi_{\pm}$ such that $\widehat{\psi_{\pm}} = F_{\pm}$. All that is left is to show that $\psi_{\pm}$ are integrable: this can be achieved by noticing that, by our construction, we have that 
\begin{equation}\label{eq:decay-special-functions} 
\|F_{\pm}\|_2 + \|F_{\pm}'\|_2 < +\infty,
\end{equation}
and hence, since $\psi_{\pm} = \widecheck{F_{\pm}}$, then \begin{align*}
\|\psi_{\pm}\|_{L^1(\R)} \le C \left( \int_{\R} |\psi_{\pm}(t)|^2 (1+t^2) \, dt \right)^{1/2} \le C'\left( \|F_{\pm}\|_2 + \|F_{\pm}'\|_2\right) < +\infty, 
\end{align*} 
as desired, concluding the proof of Theorem \ref{thm:negative-gen}. 
\end{proof}

\section{Proof of Theorems \ref{thm:main} and \ref{thm:improved}}\label{sec:counterex}

\subsection{Proof of Theorem \ref{thm:main}} We start with the `if' part. The proof below follows essentially the same arguments as in \cite{Hedenmalm-Monte-Klein-Gordon}; it has been brought to our attention that the authors of \cite{Giri-Rawat} have recently provided in \cite{Giri-Rawat-corrigendum} an analogous argument. We decided, however, to include it here for completeness. 
\begin{proof}[\unskip\nopunct] 
Indeed, suppose we have $\psi \in L^1(\R)$ such that
\begin{align}
\int_{\R} e^{-\pi i \tau (n+\theta)} \psi(\tau) \, d\tau &= 0, \,\, \forall \, n \in \Z, \label{eq:periodic}  \\ 
\int_{\R} e^{-\pi i (\beta n)/\tau} \psi(\tau) \, d \tau &= 0, \,\, 
\forall \, n \in \Z. \label{eq:inverse-periodic}  
\end{align} 
Let $M_{\theta}\psi(\tau) := e^{-\pi i \theta \tau} \psi(\tau).$ Equation \eqref{eq:periodic} is equivalent to $\sum_{j \in \Z} (M_{\theta}\psi)(2j+t) = 0$ for almost every $t \in \R.$ Analogously, changing variables $\tau \mapsto \beta/\tau$ in \eqref{eq:inverse-periodic}, we obtain that 
$$\sum_{j \in \Z} \frac{1}{(\tau+2j)^2} \psi\left( \frac{\beta}{\tau+2j}\right) = 0,$$
for almost all $t \in \R.$ Isolating the terms associated with $j = 0$ in both of these relations, we obtain 
\begin{align}
\psi(t) & = -\sum_{j \in \Z^*} e^{-2\pi i j \theta} \psi(2j + t),  \label{eq:periodic-2} \\ 
\frac{1}{t^2}\psi(\beta/t) & = -\sum_{j \in \Z^*} \frac{1}{(t+2j)^2} \psi\left( \frac{\beta}{t+2j}\right). \label{eq:inverse-periodic-2} 
\end{align}
Equivalently, \eqref{eq:inverse-periodic-2} rewrites as 
\begin{equation}\label{eq:inverse-periodic-3}
\psi(s) = - \sum_{j \in \Z^*} \frac{\beta^2}{(\beta + 2sj)^2} \psi \left( \frac{\beta s}{\beta + 2sj}\right). 
\end{equation}
Combining \eqref{eq:periodic-2} and \eqref{eq:inverse-periodic-3} yields
\begin{equation}\label{eq:functional-eq-combined} 
\psi(s) = \sum_{k,j \in \Z^*} e^{-2 \pi i j \theta} \frac{\beta^2}{(\beta+2(s+2k)j)^2} \psi \left( \frac{\beta(s+2k)}{\beta + 2(s+2k)j} \right).
\end{equation}
Define the operator $T_{\beta}:L^1[-1,1] \to L^1[-1,1]$ as 
\[
T_{\beta}f(t) = \sum_{j \in \Z^*} \frac{\beta}{(t + 2j)^2} f \left( \frac{\beta}{t + 2j} \right). 
\]
With this notation, \eqref{eq:functional-eq-combined} implies that 
\begin{equation}\label{eq:bound-psi}
|\psi(t)| \le (T_{\beta}^2 |\psi|)(t), \,\, \text{for almost all } t\in (-1,1). 
\end{equation}
From this point on, we have to distinguish between cases: first of all, if $\beta \in (0,1),$ it suffices to use Proposition 3.13.1 in \cite{Hedenmalm-Monte-Klein-Gordon}, which says that $T^{2l}_{\beta} |\psi| \to 0$ in $L^1[-1,1].$ On the other hand, if $\beta = 1,$ we only need to use \cite[Proposition~3.13.13]{Hedenmalm-Monte-Klein-Gordon}, which also guarantees to us that $\int_{-1+\delta}^{1-\delta} T^{2l}_1|\psi| \to 0$ as $l \to \infty,$ for any $\delta > 0.$ Thus, in either case, we obtain $\psi \equiv 0$ in $(-1,1).$ By \eqref{eq:inverse-periodic-3}, $\psi \equiv 0$ on $\R.$ \\

For the `only if' part, fix $\beta >1$. We simply take a $\psi_0 \in L^1, \psi_0 \not \equiv 0$, given by Theorem \ref{thm:negative-gen} such that 
\begin{align*}
\widehat{\psi_0}(\sqrt{\beta}(n+\theta)) &= 0, \, \forall \, n \in \Z, \cr
T\widehat{\psi_0}(\sqrt{\beta} n) &= 0, \, \forall \, n \in \Z. 
\end{align*}
By taking $\psi(t) = \psi_0(t/\sqrt{\beta}),$ this shows that $\Lambda_{\beta}^{\theta}$ does \emph{not} form a H.U.P. with the hyperbola, whenever $\beta > 1$, finishing the proof of that result. 
\end{proof} 

\subsection{Proof of Theorem \ref{thm:improved}}\label{sec:neg} Since Theorem \ref{thm:negative-gen} provides us, upon scaling, with the examples for the $\alpha \beta > 1$ case, we focus on the other two cases.

\begin{proof}[\unskip\nopunct]  

Applying a scaling if necessary we may assume $\alpha=\beta\le1$. Let $\psi$ be a function satisfying \eqref{eq:HUP-cond-gen} for such sequences, and define $\varphi(t) = t^{-2} \psi(1/t).$ We use the definition \eqref{eq:fourier-d-1-special} in what follows.

\vspace{2mm}

\noindent{\bf Step 1:} We first show that the conditions in Theorem \ref{thm:improved} imply that $\psi \in L^2((1+x^2)\,dx)$. Indeed, by the Poincar\'e-Wirtinger inequality \eqref{eq:p-w-set} applied to each interval $(b_k,b_{k+1})$, we have 
\[
\frac{1}{\pi^2} \int_{b_k}^{b_{k+1}} |\widehat{\varphi}'(t)|^2 \, dt \ge (b_{k+1}-b_k)^2 \frac{1}{\pi^2} \int_{b_k}^{b_{k+1}} |\widehat{\varphi}'(t)|^2 \, dt \ge \int_{b_k}^{b_{k+1}} |\widehat{\varphi}(t)|^2 \, dt.
\]
Summing the inequality above over $k \in [-N,N) \cap \Z$ yields 
\begin{align}\label{eq:poincare-prelim}
\frac{1}{\pi^2} \int_{\R} |\widehat{\varphi}'(t)|^2 \, dt \ge \frac{1}{\pi^2} \int_{b_{-N}}^{b_{N}} \left| \widehat{\varphi}'(t)\right|^2\, dt \ge \int_{b_{-N}}^{b_{N}} |\widehat{\varphi}(t)|^2 \, dt. 
\end{align}
Using Fatou's lemma on the right-hand side of \eqref{eq:poincare-prelim} then yields that $\widehat{\varphi} \in L^2$, which is equivalent after a change of variables to $x\cdot \psi(x) \in L^2,$ concluding this step. 

\vspace{2mm}

\noindent {\bf Step 2:} Suppose first $\alpha\beta < 1,$ and that $\psi \not\equiv 0$ satisfies the conditions of Theorem \ref{thm:improved}. After rescaling we may assume $\alpha=\beta<1$. Then condition \eqref{eq:HUP-cond-gen} translates to $\widehat{\psi}\left(a_{n}\right)=0=T\widehat{\psi}\left(b_{n}\right)$. We have proved in {\bf step 1} that $\widehat{\varphi} = T\widehat{\psi}$ and $\widehat{\psi}$ belong to $H^1(\R)$. Then by the Poincaré-Wirtinger inequality \eqref{eq:p-w-set}, we have
\begin{equation}\label{eq:P-W-chain} 
\int_{a_{n}}^{a_{n+1}}\left|\frac{\widehat{\psi}'(x)}{\pi}\right|^{2} d x \ge \left(a_{n+1}-a_{n}\right)^{2} \int_{a_{n}}^{a_{n+1}}\left|\frac{\widehat{\psi}'(x)}{\pi}\right|^{2} d x \geq \int_{a_{n}}^{a_{n+1}}|\widehat{\psi}(x)|^{2} d x,
\end{equation}
with equality if and only if $\widehat{\psi}=0$ on $\left(a_{n}, a_{n+1}\right)$. Summing up all these inequalities and using Plancherel twice, we obtain: 
\begin{align*}
\int_{\mathbb{R}} t^{2}|\psi(t)|^{2} d t=\int_{\mathbb{R}}\left|\frac{\widehat{\psi}^{\prime}(x)}{\pi}\right|^{2} d x& >\int_{\mathbb{R}}|\widehat{\psi}(x)|^{2} d x \cr 
=\int_{\mathbb{R}}|\psi(t)|^{2} d t&=\int_{\mathbb{R}} t^{-2}|\psi(1 / t)|^{2} d t=\int_{\mathbb{R}} t^{2}|\varphi(t)|^{2} d t .
\end{align*}
On the other hand, running the same argument as in \textbf{step 1}, we obtain that 
\[
\int_{\R} t^2 |\psi(t)|^2 \, dt < \int_{\R} t^2 |\varphi(t)|^2 \, dt,
\]
an obvious contradiction, which stems from the fact that we supposed that $\psi \not \equiv 0$, concluding the proof under the conditions of Theorem \ref{thm:improved}, since $\widehat{\psi}, T\widehat{\psi} \in H^1$ if, and only if, $\psi \in L^2((1+t^2)\,dt)$. 

\vspace{2mm}

\noindent{\bf Step 3:} Now, suppose we are in the $\alpha \beta = 1$ case. Through a dilation argument again, we may take $\alpha = \beta = 1$. By applying the proof above, we get that we must have equality in the Poincar\'e-Wirtinger inequality 
\begin{equation}\label{eq:Wirtinger-equality} 
(a_{n+1} - a_n)^2 \int_{a_n}^{a_{n+1}} \left| \frac{\widehat{\psi}'(x)}{\pi}\right|^2 \, dx \ge \int_{a_n}^{a_{n+1}} |\widehat{\psi}(x)|^2 dx,
\end{equation}
for each $n \in \Z.$ This promptly implies that there is $t_n \in \C$ such that $\widehat{\psi}(x) = t_n \cdot \sin\left( \pi \frac{x-a_n}{a_{n+1} - a_n}\right)$ for $x \in [a_n,a_{n+1}]$. 

Suppose thus that $\widehat{\psi} \not\equiv 0$. Then there is $n_0 \in \Z$ such that $t_{n_0} \neq 0$. We now focus on the next endpoint $a_{n_0 + 1}.$ Since $\psi \in \mathcal{H}_2,$ we have that $\widehat{\psi} \in C^1(\R),$ and hence $\widehat{\psi}
'(a_{n_0+1})$ is well-defined. But, since $\widehat{\psi}(x) = t_{n_0+1} \sin\left( \pi \frac{x-a_{n_0+1}}{a_{n_0 +2}-a_{n_0+1}}\right)$ for $x \in [a_{n_0+1},a_{n_0 +2}],$ by comparing the right and left limits 
\[
\lim_{r\to 0} \frac{\widehat{\psi}(a_{n_0+1}+r)-\widehat{\psi}(a_{n_0+1})}{r} = \lim_{r \to 0} \frac{\widehat{\psi}(a_{n_0+1}-r) - \widehat{\psi}(a_{n_0+1})}{-r} = \widehat{\psi}'(a_{n_0+1}), 
\]
we obtain that 
\[
-\frac{t_{n_0}}{a_{n_0+1}-a_{n_0}} = \frac{t_{n_0+1}}{a_{n_0+2}-a_{n_0+1}}.  
\]
This shows that $t_{n_0+1} \neq 0$. Since the same argument may be used for the endpoint $a_{n_0}$, we have $t_{n_0-1} \neq 0$, and hence $t_n \neq 0$ for all $n \in \Z$ by induction. Hence, $\widehat{\psi} \not\equiv 0$ on each of the intervals $[a_n,a_{n+1}]$. Looking back to the proof, we see that the inequality 
\[
 \int_{a_n}^{a_{n+1}} \left| \frac{\widehat{\psi}'(x)}{\pi}\right|^2 \, dx \ge (a_{n+1} - a_n)^2 \int_{a_n}^{a_{n+1}} \left| \frac{\widehat{\psi}'(x)}{\pi}\right|^2 \, dx
\]
must actually become an \emph{equality} for each $n \in \Z$, and hence $a_{n+1} - a_n = 1$ whenever $\widehat{\psi}' \not\equiv 0$ in $(a_n,a_{n+1})$. Repeating the endpoint derivative argument we used above for each $n \in \Z$ implies readily that 
\[
-t_n = t_{n+1}, \, \forall \, n \in \Z,
\]
which, on the other hand, shows that $\widehat{\psi}(x) = c \cdot \sin(\pi(x-\theta))$ for all $x \in \R$, for some $c \in \C \setminus\{0\}$ and some $\theta \in \R$. Since this last function is \emph{not} a Fourier transform of an $L^1$ function by Riemann-Lebesgue, we arrive at a contradiction, which stems from supposing that $\psi \not\equiv 0$, as desired. 
\end{proof}

\begin{remark} It is important to note that same endpoint results may be obtained without making use of continuity of $\widehat{\psi}'$, only using square-summability of $\psi$ and a slightly different \emph{regularity} assumption on $\psi$. Moreover, in the regular case where $A$ is a translated copy of $\Z$ and $B = \Z$, we may obtain the same result once more solely under the assumption that $\psi \in \mathcal{H}_0.$ We refer the reader to Section \ref{sec:alt-psi-proof} for a more detailed discussion on that topic. 
\end{remark}

{ 
\subsubsection{A quantitative subcritical estimate}

We now prove Theorem~\ref{thm:quantitative-sampling}. Let
\(\psi\in\mathcal H_1(\R)\), and let
\(\varphi(t)=t^{-2}\psi(1/t)\).  
The change of variables \(s=1/t\), performed separately on the two
half-lines, gives
\[
 \|\widehat{\varphi}\|_2^2=\|t\psi\|_2^2,
 \qquad
 \|t\varphi\|_2^2=\|\psi\|_2^2.
\]
More usefully, Plancherel and distributional differentiation yield
\begin{equation}\label{new-eq:trace-identities}
 \|\widehat{\psi}\|_2^2=\|\psi\|_2^2,
 \quad \|\widehat{\varphi}\|_2^2=\|t\psi\|_2^2,
 \quad \|(\widehat{\psi})'\|_2^2=\pi^2\|\widehat{\varphi}\|_2^2,
 \quad \|(\widehat{\varphi})'\|_2^2=\pi^2\|\widehat{\psi}\|_2^2.
\end{equation}

We use the following elementary interpolation estimate.

\begin{lemma}\label{new-lem:interpolation}
Let \(C=\{c_n\}_{n\in\Z}\) be strictly increasing, unbounded in both
directions, and suppose that
\[
 \inf_n(c_{n+1}-c_n)\geq\delta>0,
 \qquad \sup_n(c_{n+1}-c_n)\leq\gamma.
\]
For \(h\in L^2(\R)\cap H^1(\R)\), the sequence \((h(c_n))_n\) belongs to
\(\ell^2(\Z)\). If \(I_C(d_n)\) denotes the affine interpolant of a
sequence \((d_n)\in\ell^2(\Z)\), namely
\[
 I_C(d_n)(x)=\frac{c_{n+1}-x}{c_{n+1}-c_n}d_n+
 \frac{x-c_n}{c_{n+1}-c_n}d_{n+1}
 \quad(c_n\leq x\leq c_{n+1}),
\]
then
\begin{align}
 \sum_n|h(c_n)|^2
 &\leq \frac2\delta\|h\|_2^2+2\delta\|h'\|_2^2,
 \label{new-eq:upper-sampling}\\
 \|I_C(d_n)\|_2^2
 &\leq\gamma\sum_n|d_n|^2,
 \label{new-eq:affine-bound}\\
 \|h-I_C((h(c_n))_n)\|_2
 &\leq\frac\gamma\pi\|h'\|_2.
 \label{new-eq:affine-error}
\end{align}
\end{lemma}

\begin{proof}
For the first estimate, let \(E_n\) be the interval of length
\(\delta\) centered at \(c_n\). The intervals \(E_n\) are disjoint up
to endpoints. For \(x\in E_n\), absolute continuity and
Cauchy--Schwarz give
\[
 |h(c_n)-h(x)|^2\leq\delta\int_{E_n}|h'(s)|^2\,ds.
\]
After using \(|z+w|^2\leq2|z|^2+2|w|^2\), integrating in \(x\), and
summing over \(n\), this gives \eqref{new-eq:upper-sampling}.

On \([c_n,c_{n+1}]\), convexity gives
\[
 \int_{c_n}^{c_{n+1}}|I_C(d_n)(x)|^2\,dx
 \leq\frac{c_{n+1}-c_n}{2}(|d_n|^2+|d_{n+1}|^2).
\]
Summing proves \eqref{new-eq:affine-bound}. Finally, write \(I_Ch=I_C((h(c_n))_n)\). The function \(h-I_Ch\)
vanishes at the endpoints of every interval \([c_n,c_{n+1}]\). The
Poincar\'e--Wirtinger inequality and the fact that the derivative of
the affine interpolant is constant give
\[
 \int_{c_n}^{c_{n+1}} |h-I_Ch|^2
 \leq\frac{(c_{n+1}-c_n)^2}{\pi^2}
       \int_{c_n}^{c_{n+1}} |(h-I_Ch)'|^2
 \leq\frac{\gamma^2}{\pi^2}
       \int_{c_n}^{c_{n+1}} |h'|^2.
\]
For the last inequality, the cross term between \((h-I_Ch)'\) and the
constant derivative of \(I_Ch\) is zero, since the first function has
zero integral on the interval. Summation proves
\eqref{new-eq:affine-error}.
\end{proof}

\begin{proof}[Proof of Theorem~\ref{thm:quantitative-sampling}]
Equip \(L^2(\R)\oplus L^2(\R)\) with the norm
\[
 \|(h,k)\|_{\alpha,\beta}^2
 =\beta\|h\|_2^2+\alpha\|k\|_2^2.
\]
Apply \eqref{new-eq:affine-error} to \(\widehat{\psi}\) on \(A\) and to \(\widehat{\varphi}\) on
\(B\). In view of \eqref{new-eq:trace-identities},
\[
 \begin{aligned}
 &\beta\|\widehat{\psi}-I_A((\widehat{\psi}(a_n))_n)\|_2^2+
 \alpha\|\widehat{\varphi}-I_B((\widehat{\varphi}(b_n))_n)\|_2^2\\
 &\qquad\leq
 \frac{\beta\alpha^2}{\pi^2}\|(\widehat{\psi})'\|_2^2+
 \frac{\alpha\beta^2}{\pi^2}\|(\widehat{\varphi})'\|_2^2\\
 &\qquad=
 \alpha\beta\bigl(\beta\|\widehat{\psi}\|_2^2+\alpha\|\widehat{\varphi}\|_2^2\bigr).
 \end{aligned}
\]
The triangle inequality therefore implies
\[
 \begin{aligned}
 \|(\widehat{\psi},\widehat{\varphi})\|_{\alpha,\beta} & \leq
 \bigl\|(I_A((\widehat{\psi}(a_n))_n),I_B((\widehat{\varphi}(b_n))_n))\bigr\|_{\alpha,\beta} + \|(\widehat{\psi}-I_A((\widehat{\psi}(a_n))_n)),\widehat{\varphi}-I_B((\widehat{\varphi}(b_n))_n)\|_{\alpha,\beta}\\
 &\qquad\leq\sqrt{\alpha\beta}
 \left(\sum_n|\widehat{\psi}(a_n)|^2+\sum_n|\widehat{\varphi}(b_n)|^2\right)^{1/2} + \sqrt{\alpha\beta} \|(\widehat{\psi},\widehat{\varphi})\|_{\alpha,\beta},
 \end{aligned}
\]
where the last line follows from \eqref{new-eq:affine-bound} and the preceding estimate. Rarranging terms, squaring
and using the first two identities in \eqref{new-eq:trace-identities}
proves \eqref{new-eq:quantitative-bound}.
\end{proof}

}

\section{Applications}\label{sec:applications}

\subsection{An asymptotic version of Theorem \ref{thm:improved}}\label{sec:KNS}

As promised in the introduction, the main goal of this section will be to prove an asymptotic version of Theorem \ref{thm:improved}. In order to do so, we define the following auxiliary class of functions: 
\begin{equation}\label{eq:class-1/2}
\mathcal{C}_{1/2} =\left\{ \psi \in \mathcal{H}_0 \colon \widehat{\psi} \in L^2(dt/|t|)\right\}. 
\end{equation} 
The notation employed for this space is inspired by Proposition \ref{prop:embedding}, which shows that the $L^2(dt/|t|)$-norm is bounded by an interpolation of $L^2$ and $H^1$ norms. Our theorem then reads as follows: 

\begin{theorem}\label{thm:improved-asympt} Let $A = \{a_n\}_n$ and $B = \{b_n\}_n$ be two separated sequences. Then the following assertions hold: 
\begin{enumerate}[\normalfont(I)]
    \item Suppose $0 \in A \cup B$. If we let 
    \[
    \limsup_{n \to \pm \infty} |a_{n+1} - a_n| = \alpha, \, \limsup_{n \to \pm \infty} |b_{n+1} - b_n| = \beta, 
    \]
    then $(\Gamma,\Lambda_{\bf A,B})$ is a $\mathcal{C}_{1/2}$-\emph{H.U.P.} if $\alpha \beta < 1$. 
    
    \item If we let 
    \[
    \liminf_{n \to \pm \infty} |a_{n+1} - a_n| = \alpha, \, \liminf_{n \to \pm \infty} |b_{n+1} - b_n| = \beta, 
    \]
    then $(\Gamma,\Lambda_{\bf A,B})$ is \emph{not} a $\mathcal{C}_{1/2}$-\emph{H.U.P.} if $\alpha \beta > 1$. 
\end{enumerate}
\end{theorem}

Part (II) of the result above is a direct consequence of Theorem \ref{thm:negative-gen}, after a suitable dilation. Hence, we focus on proving Part (I) in what follows.

\subsubsection{A higher-dimensional Fourier uniqueness result} In order to prove Theorem \ref{thm:improved-asympt}, we need to use the techniques developed in \cite{Kulikov-Nazarov-Sodin} in the context of Fourier uniqueness pairs. In particular, we will need a higher-dimensional version of their main Fourier uniqueness result. 

Here, it is important to note in which sense we define vanishing: indeed, for $d=1$, all functions in $H^1 \cap L^2$ are automatically \emph{continuous}, and hence pointwise evaluation is well-defined. On the other hand, in higher dimensions this is not the case; for that reason we will say that a function $f \in L^2(\R^d) \cap H^1(\R^d)$ \emph{vanishes on a sphere of radius $r>0$} if we have that 
\[
\mathbb{T}_r(f) = 0,
\]
where $\mathbb{T}_r : W^{1,2}(B_r(0)) \to L^2(r\mathbb{S})$ denotes the Sobolev-trace operator for the ball of center 0 and radius $r>0$. 

\begin{theorem}\label{thm:KNS-higher-dim} Let $f:\R^d \to \R$ be a \emph{radial} function with $f \in L^2(\R^d) \cap H^1(\R^d)$ vanishing on a set of centered spheres of radii $\{\lambda_i\}_i,$ with $\Fd(f)$ being continuous and vanishing on centered spheres with radii $\{\gamma_i\}_i,$ where the aforementioned sequences satisfy
\begin{align*}
\limsup_{i\to \infty} \lambda_i^{p-1} |\lambda_{i+1} - \lambda_i| & < \alpha, \cr 
\limsup_{i\to \infty} \gamma_i^{q-1} |\gamma_{i+1} - \gamma_i| & < \beta,
\end{align*}
with $1/p + 1/q = 1, \, p \ge 2$ and $\alpha^{1/p} \beta^{1/q} < 1/2.$ Then $f \equiv 0.$ 
\end{theorem}

Since the $d=1$ case has been treated in \cite{Kulikov-Nazarov-Sodin}, we suppose that $d \ge 2$ throughout the proof. We begin our discussion of the proof of this result with the following simple lemma: 

\begin{lemma}\label{lemma:poincare} Let $t>0$ and let $\varepsilon \ge 1 - \left(1+\frac{1}{2t}\right)^{-\frac{d-1}{2}}.$ Let $f:\R^d \to \R$ be a function such that $f \in L^2(\R^d) \cap H^1(\R^d)$, supported outside of a centered ball of radius $1,$ which vanishes on a set of (centered) spheres with $(1-\varepsilon)(2t)^{-1}$-dense set of radii. Then, for all convex increasing $C^1$-functions $\Phi:\R_+\to\R_+,$ we have 
\[
\Phi(t^2) \int_{\R^d} |f(x)|^2 \, dx \le \int_{\R^d} \Phi(|\xi|^2) |\Fd(f)(\xi)|^2 \, d\xi.
\]
\end{lemma}

\begin{proof} The proof of this result is similar to that of \cite[Lemma~3]{Kulikov-Nazarov-Sodin}. We begin with the following: 

\begin{claim}\label{claim:optimal-radial} Let $\mathbf{C}(a,b)$ be the least of all $C>0$ such that, for the annulus $A(a,b) = \{x \in \R^d \colon a < |x| < b \},$ we have 
\begin{equation}\label{eq:Poincare}
\|f\|_{L^2(A(a,b))} \le C \|\nabla f\|_{L^2(A(a,b))},
\end{equation}
for each $f \in W^{1,2}(A(a,b))$ with $f|_{\partial A(a,b)} = 0$, in the sense of Sobolev traces. Then there exists a function $u_0$ such that 
\begin{equation}\label{eq:optimal-fct-Poincare}
\|u_0\|_{L^2(A(a,b))} = \mathbf{C}(a,b) \|\nabla u_0\|_{L^2(A(a,b))}.
\end{equation}
Furthermore, we have that any such function $u_0$ is \emph{radially symmetric}. 
\end{claim}

\begin{proof}[Proof of the Claim \ref{eq:Poincare}] The existence of such a function follows directly from the Rellich-Kondrachov theorem. For the radiality, note first that any function satisfying \eqref{eq:optimal-fct-Poincare} must satisfy the following Dirichlet-Laplace eigenvalue problem:
\begin{equation}\label{eq:Laplace} 
\begin{cases}
-\Delta u_0 = \lambda_1 u_0 & \text{ in } A(a,b), \cr 
u_0 = 0 & \text{ on } \partial A(a,b).
\end{cases}
\end{equation} 
Suppose now that $u_0$ is not radial, and consider, for a fixed rotation $R \in \text{SO}(d),$ the function 
\[
u_0^R(x) = u_0(Rx).
\]
This new function is, by direct computation, still a solution of \eqref{eq:Laplace}. Since the first eigenfunction is unique, and since, by polar coordinates, 
\[
\int_{A(a,b)} u_0 = \int_{A(a,b)} u_0^R,
\]
it follows by uniqueness of the first eigenfunction that $u_0^R = u_0$. But this is the same as saying that any function $u_0$ satisfying \eqref{eq:optimal-fct-Poincare} must be \emph{radial}, concluding the proof of the claim.
\end{proof}

Our next claim deals with a relationship between Poincar\'e constants on an annulus and the real line interval associated with it: 

\begin{claim}\label{claim:comparison-radial} Let $\mathbf{C}(a,b)$ be as in Claim \ref{claim:optimal-radial}. Let also $\mathfrak{C}(a,b)$ denote the best constant in the inequality 
\begin{equation}
\|f\|_{L^2((a,b))} \le C \|f'\|_{L^2((a,b))}.
\end{equation}
Then we have 
$$\mathbf{C}(a,b) \le \left( \frac{b}{a}\right)^{\frac{d-1}{2}} \mathfrak{C}(a,b).$$
\end{claim}

\begin{proof}[Proof of Claim \ref{claim:comparison-radial}] By Claim \ref{claim:optimal-radial}, we have that \eqref{eq:optimal-fct-Poincare} holds for $u_0.$ On the other hand, since $u_0$ is radial, we identify it with its restriction to the positive real line. Then, writing $\omega_d = \frac{\pi^{d/2}}{\Gamma(d/2+1)}$ to denote the volume of the unit ball in $\R^d$,
\[
\|u_0\|_{L^2(A(a,b))}^2 = d\cdot \omega_d \cdot \int_{a}^b |u_0(r)|^2 r^{d-1} \, dr, \,\,\,\, \|\nabla u_0\|_{L^2(A(a,b))}^2 = d \cdot \omega_d \cdot \int_a^b |u_0'(r)|^2 r^{d-1} \, dr. 
\]
Since $a^{d-1} \le r^{d-1} \le b^{d-1}$ holds in this case, then we have 
\begin{align*} 
\|u_0\|_{L^2(A(a,b))}^2 &\le d\cdot \omega_d b^{d-1} \|u_0\|_{L^2(a,b)}^2 \cr 
&\le d\cdot \omega_d b^{d-1} \mathfrak{C}(a,b)^2 \|u_0'\|_{L^2((a,b))}^2 \le (b/a)^{d-1}\mathfrak{C}(a,b)^2 \|\nabla u_0\|_{L^2(A(a,b))}^2. 
\end{align*} 
By the definition of $\mathbf{C}(a,b),$ we conclude the desired assertion.
\end{proof}

We now conclude in a similar way as in \cite[Lemma~3]{Kulikov-Nazarov-Sodin}: indeed, let $f$ be as in the statement of Lemma~\ref{lemma:poincare}. By the Poincar\'e inequality for annuli $A(r_i,r_{i+1}),$ together with Claim \ref{claim:comparison-radial}, we obtain 
\[
\int_{\R^d} |f(x)|^2 \, dx \le \left( \frac{r_{i+1}}{r_i} \right)^{d-1} \cdot \frac{(1-\varepsilon)^2}{(2\pi t)^2} \int_{\R^d} |\nabla f(x)|^2 \, dx. 
\]
By Plancherel, we obtain
\[
\int_{\R^d} t^2 |f(x)|^2 \, dx \le (1-\varepsilon)^2 \left( 1 + \frac{1}{2t}\right)^{d-1} \int_{\R^d} |\xi|^2 |\Fd(f)(\xi)|^2 \, d\xi,
\]
since $r_i > 1$ and $r_{i+1} - r_i \le (2t)^{-1}.$  By the condition on $\varepsilon,$ we prove the desired inequality with $\Phi(t) = t.$ In order to prove for general $\Phi,$ one repeats the last step in the proof of Lemma 3 in \cite{Kulikov-Nazarov-Sodin} verbatim. This finishes the proof.
\end{proof}

The next ingredient is a version of the main result in \cite[Section~5.2]{Kulikov-Nazarov-Sodin}. More specifically, we will show the following: 

\begin{proposition}\label{prop:higher-d} Let $f:\R^d \to \R$ be a function with $f \in L^2(\R^d) \cap H^1(\R^d)$ such that $f$ vanishes on centered spheres with radii $\{\lambda_i\}_i,$ and $\Fd(f)$ vanishes on centered spheres with radii $\{\gamma_i\}_i,$ where $\{\lambda_i\}_i, \{\gamma_i\}_i$ are as in the statement of Theorem \ref{thm:KNS-higher-dim}. Then we have that $f$ is a Schwartz function. More specifically, $f$ belongs to the \emph{Gelfand-Shilov class} of functions satisfying 
\begin{equation}\label{eq:f-fhat-gelfand-shilov} 
\int_{\R^d} |f(x)|^2 e^{c|x|^p} \, dx, \,\,\, \int_{\R^d} |\Fd(f)(\xi)|^2 e^{c|\xi|^q} \, d\xi < +\infty,
\end{equation} 
for some $c>0.$
\end{proposition}

\begin{proof} Since the proof follows the same lines as that of \cite[Section~5.2]{Kulikov-Nazarov-Sodin}, we only indicate the parts where changes are needed. We let then $a = p-1, b = q-1.$ By a rescaling argument we may also assume that $\alpha=\beta<1/2$ from the conditions on $\lambda_i$ and $\gamma_i$ in Theorem~\ref{thm:KNS-higher-dim}. 

We start by proving that, under the conditions of Proposition \ref{prop:higher-d}, we have $\Fd(f) \in L^2(\R^d) \cap H^1(\R^d)$. First, let $F_0 \in C^{\infty}_c(\R^d)$ be a smooth function, with $\supp(F_0) \subset B_1(0)$. Take $v \in \R^d$ sufficiently large and consider $f \cdot F_{0,v}$, where $F_{0,v}(x) = F_0(x-v)$. In order to apply Lemma~\ref{lemma:poincare} to $f \cdot F_{0,v}$, we need to show that $f \cdot F_{0,v} \in L^2(\R^d) \cap H^1(\R^d)$. Since we have that $\nabla (f \cdot F_{0,v}) = (\nabla f) \cdot F_{0,v} + f \cdot (\nabla F_{0,v})$, and since $\nabla F_{0,v}$ is a bounded function with compact support, this last claim readily follows. 

We then apply Lemma \ref{lemma:poincare} to $f\cdot F_{0,v}$ with $\Phi(t) = t$. For $|v|$ sufficiently large, $f \cdot F_{0,v}$ vanishes on a set of spheres with radii are at least $(c(|v|-1)^{a})^{-1}$-dense. This shows that
\begin{align}\label{eq:h^1-to-decay}
\int_{\R^d} |f(x)F_{0,v}(x)|^2 |x|^{2a} \, dx &\lesssim (|v|-1)^{2a} \int_{\R^d} |f(x)F_{0,v}(x)|^2 \, dx \cr
&\lesssim \int_{\R^d} |\Fd(f) * \Fd(F_{0,v}) (\xi)|^2 \, |\xi|^2 \, d \xi.
\end{align}
We then integrate both the left-hand and right-hand sides of \eqref{eq:h^1-to-decay} with respect to $v$ such that $R>|v| > X_0,$ with $X_0$ a fixed large constant. On the one hand, the integral of the left-hand side over such $v$ is at least 
\[
\int_{B_{R-1}(0) \setminus B_{X_0+1}(0)} |f(x)|^2 |x|^{2a} \, dx.
\]
On the other hand, the same integral applied to the right-hand side of \eqref{eq:h^1-to-decay} is bounded, by Plancherel, by 
\begin{equation}\label{eq:double-fourint-bound} 
\int_{\R^d \times \R^d} |\Fd(f)(\xi)|^2 |\Fd(F_{0})(\eta)|^2 |\xi+\eta|^2 \, d\xi \, d\eta. 
\end{equation} 
Since 
\[
\int_{\R^d} |\Fd(F_0)(\eta)|^2 \left( 1 + \frac{|\eta|}{|\xi|}\right)^2 \, d \eta \le C
\]
for $|\xi|$ sufficiently large, we may conclude that \eqref{eq:double-fourint-bound} is bounded by 
\[
C\int_{\R^d} |\Fd(f)(\xi)|^2 (1+|\xi|^2) \, dx < +\infty,
\]
for some $C>0$. Hence, gathering all this information, we have that 
\begin{equation}\label{eq:f-final-weight}
\int_{B_{R-1}(0) \setminus B_{X_0+1}(0)} |f(x)|^2 |x|^{2a} \, dx \le C (\|f\|_{L^2}^2 + \|f\|_{H^1}^2).
\end{equation} 
We may then take $R \to \infty$ on the left-hand side of \eqref{eq:f-final-weight} to obtain that 
\[
\int_{\R^d} |f(x)|^2 |x|^{2a} \, dx < +\infty.
\]
Since $a = p-1\ge 1$, this shows that $\Fd(f) \in L^2(\R^d) \cap H^1(\R^d)$, as desired. 

We then note that the assertion that each $f$ satisfying the statement of Theorem \ref{thm:KNS-higher-dim} must belong to the Schwartz class can be proved by essentially simply following the proof above, now employing the same argument on both spatial and frequency sides, and replacing the weight used by a sequence of weights $\Phi_t^p(s)$ which agree with $|s|^p$ on a large set, are convex and grow linearly at infinity. We shall skip the details in that proof, and refer the reader to \cite[Section~5.1]{Kulikov-Nazarov-Sodin} for further details. 

For the proof of \eqref{eq:f-fhat-gelfand-shilov}, a slightly more precise argument is needed. With that in mind, let
$$F = 1_{B_u} * \underbrace{\avgI_{B_{u/2k}} * \cdots *  \avgI_{B_{u/2k}}}_{k-\text{fold convolution}} \quad ,$$
where we define $\avgI_B(x) := \frac{1}{|B|} 1_B(x),$ and $B_r$ denotes the Euclidean ball with center $0$ and radius $r$. Note that, for a function $f$ as in the statement of the proposition, there exists $X_0 > 1$ such that, if $|x|>X_0,$ then $|x|$ belongs to an interval of length at most $(2\sigma |x|^{p-1})^{-1},$ where $\sigma > 1$ and the endpoint of such interval are radii of spheres where $f$ vanishes. With that in mind, take $v$ to be a vector in $\R^d$ such that $|v| - \frac{3}{2} u > X_0.$ Then the function $f \cdot F_v,$ where $F_v(x) = F(x-v),$ satisfies the hypotheses of Lemma \ref{lemma:poincare} with $t = \frac{1+\sigma}{2} (|v|-\frac{3}{2} u)^{2a} = \tilde{\sigma}(|v|-\frac{3}{2}u)^{2a},$ since we can take $\varepsilon \gtrsim \sigma - 1$ there. Hence, using $\Phi(t) = t^{\theta},$ we obtain that 
\[
\left(|v|-\frac{3}{2}u\right)^{2a\theta} \int_{\R^d} |f(x)|^2 |F_v(x)|^2 \, dx \le \tilde{\sigma}^{-2\theta} \int_{\R^d} |\xi|^{2\theta} | (\Fd(f) * \Fd(F_v)(\xi)|^2 \, d\xi. 
\]
Again, we integrate both sides over $v \in \R^d$ such that $|v| > \frac{3}{2} u + X_0.$ On the left-hand side, we bound it from below, in analogy to the one-dimensional case, by 
\begin{equation}\label{eq:lower-1}
\left( \frac{K-4}{K} \right)^{2a\theta} \int_{|x| \ge K u} |x|^{2a\theta} |f(x)|^2 \, dx \cdot \left( \int_{\R^d} |F(y)|^2 \, dy\right),
\end{equation}
where $K>4$ and $u \ge X_0.$ For the right-hand side, we bound it from above by 
\begin{align}\label{eq:upper-1}
\int_{\R^d} \left( \int_{\R^d} |(\Fd(f) * \Fd(F_v) (\xi)|^2 \, |\xi|^{2\theta} \, d\xi \right) \, dv = \int_{\R^d \times \R^d} |\Fd(f)(\xi)|^2 |\Fd(F)(\eta)|^2 |\xi+\eta|^{2\theta} \, d \xi \, d\eta. 
\end{align}
We now need to estimate the contribution on the integral above stemming from the Fourier transform of $F.$ Effectively, the Fourier transform of $\avgI_{B_1}$ is $\Fd\left(\avgI_{B_1}\right)(\xi) = \frac{1}{\omega_d} \cdot |\xi|^{-d/2} J_{d/2} (2\pi |\xi|),$ where $J_{\nu}$ denotes the Bessel function of order $\nu.$ Hence, since $\avgI_{B_r}(x) = r^{-d} \avgI_{B_1}(x/r)$ for all $r>0$,
\[
\Fd(F)(\eta) = u^{d/2} \frac{J_{d/2}(2\pi \cdot u|\eta|)}{|\eta|^{d/2}} \, \cdot \left( \frac{J_{d/2}(2\pi \cdot (u/2k) \cdot |\eta|)}{\omega_d \cdot | (u/2k) \cdot \eta|^{d/2}} \right)^k. 
\]
We now claim that
\[
(|\xi| + |\eta|) \cdot \frac{ |J_{d/2}( \pi \frac{u}{k}|\eta|)|}{\omega_d \cdot |\frac{u}{2k} \eta|^{d/2}} \le |\xi| + \sqrt{\frac{2d + 4}{\pi}} \frac{k}{u \pi}.
\]
Indeed, since $\left|\Fd\left(\avgI_{B_1}\right)\right| \le 1,$ the inequality is trivially true if $|\eta| \le \sqrt{\frac{2d + 4}{\pi}} \frac{k}{u \pi}.$ If, on the other hand, $|\eta| > \sqrt{\frac{2d + 4}{\pi}} \frac{k}{u \pi},$ we write the Bessel function $J_{d/2}(t)$ as 
\[
J_{d/2}(t) =\frac{(d-1) \cdot\left(\frac{1}{2} t\right)^{\frac{d}{2}-1}}{\sqrt{\pi} \Gamma\left(\frac{d+1}{2}\right)}\int_0 ^{\frac{1}{2} \pi} \sin (t \cos \theta) \sin ^{d-2} \theta \cos \theta \, d \theta;
\]
see, for instance, \cite[(7), \S III.3.3, p.~48]{watson1995treatise}. From this formula, it follows that 
\begin{equation}\label{eq:bound-J-h-d} 
2^{d/2} \Gamma\left( \frac{d}{2} + 1 \right) \frac{|J_{d/2}(t)|}{t^{d/2-1}} \le \sqrt{\frac{2d + 4}{\pi}},
\end{equation} 
upon using Gautschi's inequality \cite[(5.6.4)]{NIST:DLMF}, which asserts that $\Gamma(x+1) \le (x+1)^{1-s} \cdot \Gamma(x+s)$ for all $x \ge 0, \, s \in (0,1)$, and the fact that 
\[
(d-1) \left| \int_0^{\frac{\pi}{2}} \sin(t \cos \theta) \, \sin^{d-2}\theta \, \cos \theta \, d \theta \right| \le \int_0^{\frac{\pi}{2}} |(\sin^{d-1})'(\theta)| \, d\theta = 1. 
\]
It follows hence from \eqref{eq:bound-J-h-d} that 
\begin{align*} 
(|\xi| + |\eta|) \cdot \frac{ |J_{d/2}( \pi \frac{u}{k}|\eta|)|}{\omega_d \cdot |\frac{u}{2k} \eta|^{d/2}} & \le  |\xi| + |\eta| \cdot \left( 2^{d/2} \Gamma\left( \frac{d}{2} + 1\right) \frac{|J_{d/2}(\pi \frac{u}{k} |\eta|)| }{|\pi \frac{u}{k} \eta|^{d/2}}\right) \cr 
    & \le |\xi| + \sqrt{\frac{2d + 4}{\pi}} \frac{k}{u \pi}. 
\end{align*}

Therefore, using that the Fourier transform of the (normalized) unit ball in dimension $d$ is uniformly bounded by 1, for $k\ge \theta$, we conclude that the right-hand side of \eqref{eq:upper-1} is bounded by
\[
\int_{\R^d \times \R^d} |\Fd(f)(\xi)|^2 \left( |\xi| +\sqrt{\frac{2d + 4}{\pi}} \frac{k}{u \pi}\right)^{2\theta} \left| \Fd\left(1_{B_u} \right) (\eta)\right|^2 \, d\xi d\eta 
\]
\[
= \omega_d \cdot u^d \int_{\R^d} |\Fd(f)(\xi)|^2  \left( |\xi| + \sqrt{\frac{2d + 4}{\pi}} \frac{k}{u \pi} \right)^{2\theta} \, d\xi. 
\]
On the other hand, taking into account that $\int_{\R^d} |F(y)|^2 \, dy \ge \omega_d \cdot 2^{-d} u^d$ -- since $F \equiv 1$ in $B_{u/2}$ -- in \eqref{eq:lower-1}, we obtain the improved lower bound
\begin{equation}\label{eq:lower-2}
\omega_d \cdot 2^{-d} u^d \left( \frac{K-4}{K} \right)^{2a\theta} \int_{|x| \ge K u} |x|^{2a\theta} |f(x)|^2 \, dx. 
\end{equation}
Putting all these considerations together, we obtain 
\begin{equation}\label{eq:comparison-2}
\left( \frac{K-4}{K} \right)^{2a\theta} \int_{|x| \ge K u} |x|^{2a\theta} |f(x)|^2 \, dx \le 2^{d}  \tilde{\sigma}^{-2\theta} \int_{\R^d} |\Fd(f)(\xi)|^2  \left( |\xi| + \sqrt{\frac{2d + 4}{\pi}} \frac{k}{u \pi} \right)^{2\theta} \, d\xi. 
\end{equation}
At this point, we can just repeat the last part of the proof of the one-dimensional case verbatim. Effectively, we rewrite \eqref{eq:comparison-2} as 
\begin{equation}\label{eq:comparison-3}
 \int_{|x| \ge K u} |x|^{2a\theta} |f(x)|^2 \, dx \le 2^{d}  \left( \frac{1}{\tilde{\sigma}} \left(\frac{K}{K-4}\right)^{a} \right)^{2\theta}\int_{\R^d} |\Fd(f)(\xi)|^2  \left( |\xi| + \sqrt{\frac{2d + 4}{\pi}} \frac{k}{u \pi}  \right)^{2\theta} \, d\xi. 
\end{equation}
We choose $K$ large enough so that $\frac{1}{\tilde{\sigma}} \left( \frac{K}{K-4}\right)^a < 1,$ take $k \in [\theta, \theta \pi],$ and bound the integral on the right-hand side of \eqref{eq:comparison-3} by 
 \[
 \left(\frac{K+1}{K} \right)^{2\theta} \int_{\R^d} |\Fd(f)(\xi)|^2 |\xi|^{2\theta}\, d\xi + \int_{|\xi| \le K \sqrt{\frac{2d + 4}{\pi}} \frac{k}{u \pi} } |\Fd(f)(\xi)|^2  \left( |\xi| + \sqrt{\frac{2d + 4}{\pi}} \frac{k}{u \pi}  \right)^{2\theta} \, d\xi 
 \]
 \vspace{2mm}
 \[
    \le  \left(\frac{K+1}{K} \right)^{2\theta} \int_{\R^d} |\Fd(f)(\xi)|^2 |\xi|^{2\theta}\, d\xi + \left( \sqrt{\frac{2d + 4}{\pi}} (K+1)\right)^{2\theta} \left( \frac{\theta}{u} \right)^{2\theta} \int_{\R^d} |f(x)|^2 \, dx. 
\]
Since 
\[
\int_{|x| \le K u} |f(x)|^2 |x|^{2a\theta} \, dx \le K^{2a\theta} u^{2a\theta} \int_{\R^d} |f(x)|^2 \, dx,
\]
inserting these considerations into \eqref{eq:comparison-3} we obtain that 
\begin{align}
\int_{\R^d} |x|^{2a\theta} |f(x)|^2 \, & dx  \le 2^{d} \left( \frac{1}{\tilde{\sigma}} \left(\frac{K}{K-4}\right)^{a} \cdot \frac{K+1}{K} \right)^{2\theta} \int_{\R^d} |\Fd(f)(\xi)|^2 |\xi|^{2\theta} \, d\xi\cr
 & + \left(K^{2a\theta} u^{2a\theta} + 2^{d+1} \left( \sqrt{\frac{2d + 4}{\pi}} (K+1)\right)^{2\theta} \left( \frac{\theta}{u} \right)^{2\theta} \right) \int_{\R^d} |f(x)|^2 \, dx. 
\end{align}
Using now that the factor in front of the first integral on the right-hand side above converges to zero as $\theta \to \infty$, and taking $u = \theta^{1/(a+1)}$, which is allowed since $X_0$ is fixed, we get 
\begin{equation}
\int_{\R^d} |f(x)|^2 |x|^{2a\theta} \, dx \le \frac{1}{2} \int_{\R^d} |\Fd(f)(\xi)|^2 |\xi|^{2\theta} \, d\xi + C_d^{2\theta} \theta^{\frac{2a}{a+1} \theta}\int_{\R^d} |f(x)|^2 \, dx.
\end{equation}
Now, notice that the exact same can be done if we switch the roles of $f$ and $\widehat{f}.$ By doing so, if we let $\kappa = a \theta,$ we obtain 
\[
\int_{\R^d} |x|^{2a\theta} |f(x)|^2 \, dx + \int_{\R^d} |\xi|^{2b\kappa} |\Fd(f)(\xi)|^2 \, d\xi \le C_d^{\theta} \theta^{\frac{2a}{a+1} \theta} \, \int_{\R^d} |f(x)|^2 \, dx. 
\]
Calling $\theta = \frac{a+1}{2a} \ell,$ where $\ell \in \Z_+,$ and reverting back $a = p-1, b = q-1,$ we obtain 
\[
\int_{\R^d} |x|^{p \ell} |f(x)|^2 \, dx + \int_{\R^d} |\xi|^{q \ell} |\Fd(f)(\xi)|^2 \, d\xi \le C_d^{\ell} \ell! \, \int_{\R^d} |f(x)|^2 \, dx. 
\]
Dividing both sides by $(2C_d)^\ell\ell!$ and summing over all $\ell\ge0$ then yields our claim. 
\end{proof}

\begin{proof}[Proof of Theorem \ref{thm:KNS-higher-dim}] From Proposition \ref{prop:higher-d}, we wish to conclude, in analogy to \cite[Section~5.3]{Kulikov-Nazarov-Sodin} and using the methods from \cite[Section~5.2.3,~Step~2]{Ramos-Sousa}, that $f$ and $\Fd(f)$ may be extended---when identified with their radial representation---as entire functions of order $p$ and $q$, respectively. 

Indeed, it readily follows from the techniques from \cite{Kulikov-Nazarov-Sodin} that for some $c'>0,$ one has $|f(x)| \lesssim e^{-c'|x|^p}, \, |\Fd(f)(\xi)| \lesssim e^{-c' |\xi|^q},$ for each $x, \xi \in \R^d.$  In order to conclude the assertion on analyticity, we follow the overall structure already present in \cite[Section~5.3]{Kulikov-Nazarov-Sodin}. As a matter of fact, our first observation is that, since we can represent 
\begin{equation}\label{eq:radial-1}
\Fd(f)(\xi) = (2\pi)^{d/2} \int_0^{\infty} f(r)\cdot r^{d-1} \mathcal{A}_{\frac{d-2}{2}} (|\xi| r) \, dr, 
\end{equation}
where $\mathcal{A}_{\nu}(s) = (2\pi s)^{\nu} J_{\nu}(2 \pi s)$, and since $\mathcal{A}_{\nu}$ is an even analytic function of its argument, which satisfies, when extended to the whole complex plane, 
\begin{equation}\label{eq:kernel}
|\mathcal{A}_{\nu}(\xi + i \eta)| \le C_d e^{2 \pi |\eta|},
\end{equation}
the decay conditions from $f, \Fd(f)$ above imply that $\Fd(f)$ and $f$ can be extended (when regarded as one-variable functions) to the whole complex plane as entire functions of orders $p$ and $q,$ respectively. With this in mind, we let 
\[
\kappa_1(\theta) = \frac{1}{2\pi} \limsup \frac{\log |f(re^{i\theta})|}{r^p}, 
\]
\[
\kappa_2(\theta) = \frac{1}{2\pi} \limsup \frac{\log |\Fd(f)(\rho e^{i\theta})|}{\rho^q}
\]
denote the \emph{Phragm\'en-Lindel\"off indicators} of $f$ and $\Fd(f)$, respectively. We then denote $K_i = |\kappa_i(0)|.$ For these indicators, the following Claim---which is a version of \cite[Claim~2]{Kulikov-Nazarov-Sodin}---holds: 

\begin{claim}\label{claim:lower-bound-indicators}
We have 
\begin{align}
\kappa_1(\theta) \le \frac{1}{p (q K_2)^{p/q}} |\sin \theta|^p, \cr 
\kappa_2(\varphi) \le \frac{1}{q (p K_1)^{q/p}} |\sin \varphi |^q,
\end{align}
for all $\theta,\varphi \in [0,2\pi).$
\end{claim}

\begin{proof} The only difference between the proof of this result and that of \cite[Claim~2]{Kulikov-Nazarov-Sodin} is that the current one must resort to \eqref{eq:radial-1} and \eqref{eq:kernel}. Once these two formulas are applied, the rest of the proof in \cite[Claim~2]{Kulikov-Nazarov-Sodin} can be repeated verbatim, and hence we omit that part. 
\end{proof}

We then state the following crucial claim:

\begin{claim}\label{claim:ineq-indicators} With the definitions above, we have that, for $\theta \in \left(0, \frac{\pi}{2p}\right),$ 
\begin{equation}\label{eq:indicator-1} 
K_1 \ge \frac{1}{p} \tan(p\theta) - \frac{1}{p(qK_2)^{p/q}} \frac{(\sin \theta)^p}{\cos(p\theta)}
\end{equation}
and, for $\varphi \in \left(0, \frac{\pi}{2q}\right),$
\begin{equation}\label{eq:indicator-2}
K_2 \ge \frac{1}{q} \tan(q\varphi) - \frac{1}{q(pK_1)^{q/p}} \frac{(\sin \varphi)^q}{\cos(q \varphi)}.
\end{equation}
\end{claim}

\begin{proof}[Proof of Claim \ref{claim:ineq-indicators}] The proof is a verbatim adaptation of \cite[Equation~(5.3.1)]{Kulikov-Nazarov-Sodin}, and hence we omit it. \qedhere
\end{proof}

In order to conclude the proof of Theorem \ref{thm:KNS-higher-dim}, we use the inequality $\frac{\tan p \theta}{p} > \frac{\sin \theta}{(\cos p\theta)^{1/p}}, \theta \in (0,\pi/2p),$ in the statement of Claim \ref{claim:ineq-indicators}. Letting $X = \frac{\sin \theta}{(\cos p\theta)^{1/p}},$ we get from these considerations that 
\[
K_1 > X - \frac{1}{(qK_2)^{p/q}} \frac{X^p}{p}, \, \forall X > 0.  
\]
A direct optimization argument concludes that the maximum of the right-hand side of the expression above is at $X = q K_2,$ for which we obtain $K_1 > K_2.$ On the other hand, we may exchange the roles of $f$ and $\Fd(f)$ in the proof above, which allows us to obtain $K_2 > K_1,$ which is an obvious contradiction. This contradiction arises when we suppose that $f \not\equiv 0,$ which finishes the proof. 
\end{proof}

\begin{remark} As it turns out, Theorem \ref{thm:KNS-higher-dim} has consequences for \emph{arbitrary} functions vanishing on spherical shells, rather than only radial ones, as explicitly stated in the next result. 

\begin{corollary}\label{cor:uniquness-general-high-dim} Let $f, \Fd(f) \in L^2(\R^d) \cap H^1(\R^d)$. Suppose, moreover, that 
$$f|_{\lambda_j \cdot \mathbb{S}^{d-1}} = \Fd(f)|_{\gamma_j \cdot \mathbb{S}^{d-1}} \equiv 0, \, \, \forall j \ge 1,$$
where $\{\lambda_j\}_j, \{\gamma_j\}_j$ are as in Theorem \ref{thm:KNS-higher-dim}. Then $f \equiv 0.$
\end{corollary}

The proof of that result follows from using Proposition \ref{prop:higher-d}, which ensures that $f \in \mathcal{S}(\R^d)$ under the hypotheses of Corollary \ref{cor:uniquness-general-high-dim}, together with Theorem \ref{thm:KNS-higher-dim} in conjunction with the following result, originally proved by M. Stoller \cite{Stoller}: 

\begin{theorem*}[Corollary~2.2 in \cite{Stoller}] Fix a dimension $d \geq 2$ and fix two subsets $R, \hat{R} \subset(0, \infty)$. Suppose that for all $p \in\left\{d+2 m: m \in \mathbb{N}_0\right\}$ an all radial Schwartz functions $f$ defined on $\R^p$, the following implication holds:
$$
\left(\left.f\right|_{\cup_{r \in R} r S^{p-1}}=0 \text { and }\left.\Fd(f)\right|_{\cup_{\rho \in \hat{R}} \rho S^{p-1}}=0\right) \Longrightarrow f=0.
$$
Then the same implication holds for arbitrary $f \in \mathcal{S}\left(\mathbb{R}^d\right)$.
\end{theorem*}

\end{remark}

\subsubsection{Proof of Theorem \ref{thm:improved-asympt}} We are now able to move on to the proof of Part (I) of Theorem \ref{thm:improved-asympt}: 

\begin{proof}[Proof of Theorem \ref{thm:improved}, Part \normalfont{(I)}] Suppose we have a $\psi \in \mathcal{C}_{1/2}$ satisfying the conditions in Part (I) of Theorem \ref{thm:improved}. Let $\Theta_+(u) = \frac{\theta(|u|^2)}{|u|^2}$. By the reductions in Section \ref{sec:equivalences}, and by \eqref{eq:T-Fourier-dim-4}, the following estimate holds: 
\begin{align}
\int_{\R^4}|\Theta_+(u)|^2(1+|u|^2) \, du &= c \cdot \int_0^{\infty} \frac{|\theta(s^2)|^2}{s^4}(1+s^2) s^3 \, ds =\frac{c}{2} \int_{\R_+} |\theta(t)|^2 \frac{1+t}{t} \, dt \cr 
&\lesssim \|\theta\|_{L^2}^2 + \|\theta\|_{L^2(dt/|t|)}^2.  
\end{align}
Let $A_+ = A \cap (0,+\infty), \, B_- = B \cap (-\infty,0).$ Since $\Theta_+(\sqrt{a}) = \mathcal{F}_4(\Theta_+)(\sqrt{-b}) = 0$, for all $a \in A_+, \, b \in B_-$, and since 
\begin{align*}
\limsup_{n\to \infty} \sqrt{a_{n+1}} |\sqrt{a_{n+1}} - \sqrt{a_n}|\le \frac{1}{2} \limsup_{n \to \infty} |a_{n+1} - a_n| < \frac{\alpha}{2},
\end{align*}
and, by the same token, 
\begin{align*}
\limsup_{n \to - \infty} \sqrt{-b_n} |\sqrt{-b_{n+1}}- \sqrt{-b_n}| < \frac{\beta}{2},
\end{align*}
the sequences of radii $\gamma_i = \sqrt{a_i},$ for $a_i > 0$, $\lambda_i = \sqrt{-b_i}$ for $b_i < 0$ satisfy the hypotheses of Theorem~\ref{thm:KNS-higher-dim} with $p = q =2$. Moreover, since $\Theta_+$ is a continuous function by hypothesis, we are in a perfect position in order to apply Theorem \ref{thm:KNS-higher-dim}. It then plainly implies that $\Theta_+ \equiv 0$. Since the same can be done for the negative part of $\theta$, we have $\theta \equiv 0$, finishing the proof. 
\end{proof}

\begin{remark} As mentioned after the statement of Theorem \ref{thm:improved-asympt}, the definition of the class $\mathcal{C}_{1/2}$ is inspired by Proposition \ref{prop:embedding}: indeed, if one assumes the stronger property that $\theta \in L^2\cap H^1$, then it follows directly that $T\theta \in L^2 \cap H^1$ by Proposition \ref{prop:properties-T}. By that same result and Proposition \ref{prop:embedding}, we have promptly that $T\theta,\theta \in L^2 \cap L^2(dt/|t|)$, which is what we end up using in the proof of Part (I) of Theorem \ref{thm:improved-asympt}. 

This highlights the fact that the main results in this manuscript deal with many classes of functions that are \emph{different} in comparison to the ones in the previous literature. Indeed, the pioneering article \cite{Hedenmalm-Montes-Annals} and subsequent works such as \cite{Hedenmalm-Monte-Klein-Gordon} all deal with $\psi \in L^1(\R)$, which automatically implies that $\theta = \widehat{\psi}$ and $T\theta$ are both continuous functions. Our results here, however, assume only $\psi \in L^2$ in terms of integrability, and thus the condition $\widehat{\psi} \in C(\R)$ is needed in order to make sense of pointwise values of $\theta$. 

To that extent, we remark that there is a vast array of functions $\psi \not\in L^1(\R)$ satisfying the conditions above: fix for instance any $f_0 \in C(\R)$ $2$-periodic such that its sequence of Fourier series coefficients is \emph{not} in $\ell^1(\Z)$. A simple example of such a function is given by \[
f_0(x) = \sum_{n=1}^{\infty} \frac{1}{n^2} \sin\left( (2^{n^3} + 1) \pi |x|\right),\qquad x\in [-1,1]\,,
\]
extended 2-periodically, as first observed by F\'ejer \cite{Merx,Turan1970}. Then consider 
\[
\psi(x) = \sum_{n \in \Z} \widehat{f_0}(n) \cdot e^{\pi i (x-n)} 1_{[n-1,n+1]}(x). 
\]
By Plancherel's theorem, it follows easily that $\psi \in L^2(\R)$, but $\psi \not\in L^1(\R)$. Moreover, we have that 
\[
\widehat{\psi}(\xi) = 2\left( \sum_{n \in \Z} \widehat{f_0}(n) e^{\pi i n \xi}\right) \cdot \frac{\sin(\pi (\xi-1))}{\pi (\xi-1)} = 2\cdot f_0(\xi) \cdot \frac{\sin(\pi(\xi-1))}{\pi (\xi-1)}. 
\]
By construction, we have that $\widehat{\psi}$ is continuous and vanishes at infinity, as desired. Moreover, this construction incidentally yields a function in the class $\mathcal{C}_{1/2}$ which is not in $L^1$. Indeed, we have $|\widehat{\psi}(\xi)| \le |\xi|$ for small $\xi$, and hence $\widehat{\psi} \in L^2(d\xi/|\xi|)$ as a consequence, which further emphasizes the fundamental difference between the results presented here and the ones in \cite{Hedenmalm-Montes-Annals,Hedenmalm-Monte-Klein-Gordon,Hedenmalm-Montes-Hinfty,CHMR} and other related work. 
\end{remark} 
 
\subsection{One-sided information on space and frequency}\label{sec:one-sided} Although we stated Theorem \ref{thm:improved} as a recovery result for perturbed lattice crosses, we can actually obtain perturbed versions of the \emph{one-sided} results from \cite{Hedenmalm-Monte-Klein-Gordon}. In order to state those, we will need a bit of additional notation. 

\begin{definition} For a set $\Lambda \subset \R^2$, we say that $\psi \in \mathcal{H}_\ell$ belongs to the class $\mathcal{H}_\ell(\Lambda)$ if, for any $\xi = (\xi_1,\xi_2) \in \Lambda$, we have 
\begin{equation}
E\psi(\xi_1,\xi_2) = \int_{\R} \psi(t) e^{\pi i (t \xi_1 + \xi_2/t)} \, dt = 0. 
\end{equation}
\end{definition}

For that class, we define in analogy to \cite[Definition~1.2.1]{Hedenmalm-Monte-Klein-Gordon} the \emph{$\mathcal{H}_\ell$-Zariski closure} of $\Lambda$ as 
\[
\text{zclos}_2^{\ell}(\Lambda) = \left\{ \xi \in \R^2 \colon \forall\, \psi \in \mathcal{H}_\ell(\Lambda), \, E\psi(\xi) =0\right\}. 
\]
We are now able to state our next result. Recall that, for $A, B \subset \R$, we define $\Lambda_{\bf A,B} = ( A \times \{0\}) \cup (\{0\} \times B).$ 

\vspace{-2mm}

\begin{theorem}\label{thm:one-side} Let $A \subset \R_{\ge 0},B\subset \R_{\le 0}$ be two uniformly discrete sets. Then: 
\begin{enumerate}[\normalfont(I)]
    \item Suppose that 
$$\sup_n |a_{n+1} - a_n| = \alpha, \, \, \sup_n |b_{n+1} - b_n| = \beta,$$
with $\alpha \beta < 1$. Then $\normalfont{\text{zclos}}_2^0(\Lambda_{\bf A,B}) = \R_{\ge 0} \times \R_{\le 0}$. Moreover, if $\alpha\beta = 1$, $\normalfont{\text{zclos}}_2^2(\Lambda_{\bf A,B})= \R_{\ge 0} \times \R_{\le 0}$ as well. 

\vspace{2mm}

    \item Suppose that 
$$\inf_n |a_{n+1} - a_n| = \alpha, \, \, \inf_n |b_{n+1} - b_n| = \beta,$$
with $\alpha \beta > 1$. Then $\R_{\ge 0} \times \R_{\le 0} \setminus \normalfont{\text{zclos}}_2^2(\Lambda_{\bf A,B})$ contains infinitely many points. 
\end{enumerate}
\end{theorem}

\begin{proof} We start with the case $\alpha \beta < 1$ in Part (I). Without loss of generality, $\alpha = \beta < 1$. Let $\psi \in \mathcal{H}_0$ be a function such that
\begin{equation}\label{eq:vanishing-partial-cross} 
E\psi(a_n,0) = E\psi(0,b_k) = 0, \, \forall \, n, k \in \N. 
\end{equation}
The crucial observation we shall use throughout is that we may \emph{replace} $\psi$ by $\tilde{\psi}(t) = \frac{\psi - i \mathbf{H}\psi}{2}$ so that \eqref{eq:vanishing-partial-cross} still holds for $\tilde{\psi}$, where $\mathbf{H}\psi(t)$ denotes the \emph{Hilbert transform} of $\psi$ at $t$, defined as 
\begin{equation}\label{eq:def-Hilbert-transform}
\mathbf{H}\psi(t) = \frac{1}{\pi} \text{p.v.} \int_{\R} \frac{\psi(x)}{t-x} \, dx.
\end{equation}
Indeed, note first of all that $\tilde{\psi} \in L^2$. Moreover, if we simply use the property that 
\[
\widehat{\mathbf{H}\psi}(\xi) = -i \sgn(\xi) \widehat{\psi}(\xi),
\]
we obtain that
\[
E\tilde{\psi}(a_n,0) = \frac{1}{2} \left( \int_{\R} \psi(t) e^{-\pi i t a_n} - i \int_{\R} \mathbf{H}\psi(t)\, e^{-\pi i t a_n} \, dt\right) = \theta(a_n) = 0, \, \forall n \ge 0. 
\]
For the second equality, we need the following observation: 
\begin{align}\label{eq:Hilbert-invariance}
\mathbf{H}\psi(1/t) &= \frac{1}{\pi} \text{p.v.} \int_{\R} \frac{\psi(x)}{\frac{1}{t}-x} \, dx = \frac{t}{\pi} \text{p.v.} \int_{\R} \frac{\psi(x)}{1-t \cdot x} \, dx \cr 
& = \frac{t^2}{\pi} \, \text{p.v.} \int_{\R} \frac{\psi(x) \cdot x}{1-t \cdot x} \, dx + \frac{t}{\pi} \int \psi = \frac{t^2}{\pi} \text{p.v.} \int_{\R} \frac{\psi(1/s)\cdot (1/s) }{1- \frac{t}{s}} \frac{ds}{s^2} \cr 
& = -\frac{t^2}{\pi} \text{p.v.} \int_{\R} \left( \frac{1}{s^2} \psi\left( \frac{1}{s}\right) \right) \, \frac{ds}{t-s} = - t^2 \mathbf{H}\varphi(t), 
\end{align} 
where $\varphi(s) = \psi(1/s) \cdot (1/s^2)$, and where we used that $\int \psi = 0$. Hence, 
\begin{align*}
    \int_{\R} \frac{1}{t^2} \mathbf{H}\psi(1/t) e^{- \pi i b_n t} \, dt = - \int_{\R} \mathbf{H}\varphi(t) \, e^{-\pi i b_n t} \, dt, 
\end{align*}
and thus we obtain $\widehat{\mathbf{H}\varphi} \in C(\R \setminus \{0\})$, and also 
\[
E\tilde{\psi}(0,b_n) = \frac{1}{2} \left( \int_{\R} \varphi(t) e^{-\pi i b_n t}\, dt + i \int_{\R} \mathbf{H}\varphi(t)  e^{-\pi i b_n t} \, dt \right) = T\theta(b_n) = 0, \, \forall \, n \ge 0,
\]
again by the properties of the Hilbert transform, as $b_n < 0$ for all $n$. Hence, we may suppose that $\tilde{\psi}$ is of the form above in what follows, and as a consequence we may always assume that $\widehat{\psi}$ is supported on $\R_{\ge 0}$. 

The rest of the proof is similar to that of Theorem \ref{thm:improved}: we start by once more observing that the Poincar\'e-Wirtinger argument of \eqref{eq:poincare-prelim} yields again that $\widehat{\varphi} \in L^2(\R)$, and hence $\psi \in L^2(x^2 \, dx)$. Applying now the same argument as in \eqref{eq:P-W-chain}, we have 
\begin{equation}\label{eq:P-W-bound-one-side} 
\int_{\R} |x|^2 |\psi(x)|^2\, dx = \int_{\R} \left| \frac{\widehat{\psi}'(x)}{\pi}\right|^2 \, dx = \int_0^{\infty} \left| \frac{\widehat{\psi}'(x)}{\pi}\right|^2 \, dx \ge \int_0^{\infty} |\widehat{\psi}(x)|^2 \, dx = \int_{\R} |\psi(x)|^2 \, dx. 
\end{equation} 
On the other hand, the same argument can be applied to $\varphi$ by \eqref{eq:Hilbert-invariance}: indeed, we have that $\widehat{\varphi} = T\theta.$ Since if $\supp(\theta) \subset \R_{\ge 0}$, we have by Proposition \ref{prop:properties-T} that $\supp(T\theta) \subset \R_{\le 0}$, the same Poincar\'e-Wirtinger argument works to show that 
\begin{equation}\label{eq:P-W-bound-other-side}
\int_{\R} |x|^2 |\varphi(x)|^2 \, dx \ge \int_{\R} |\varphi(x)|^2 \, dx.
\end{equation}
We then must have equality in the inequalities \eqref{eq:P-W-bound-one-side} and \eqref{eq:P-W-bound-other-side}, which implies equality in \eqref{eq:P-W-chain}. If $\alpha = \beta < 1$, this is only possible if $\psi \equiv 0$ whenever $\psi $ has Fourier support on $\R_{\ge 0}$. Equivalently, undoing the change $\psi \mapsto \tilde{\psi},$ we have that 
\[
\int_{\R} \psi(t) e^{\pi i \eta/t} \, dt =  \int_{\R} \psi(t) e^{-\pi i t \xi} \, dt = 0, \, \forall \, \xi, \eta \ge 0. 
\]
By \cite[Proposition~1.7.2]{Hedenmalm-Monte-Klein-Gordon}, we conclude that 
\[
\int_{\R} \psi(t) e^{\pi i (\eta/t - \xi t)} \, dt = 0, \, \forall \xi,\eta \ge 0, 
\]
which is the desired claim in that case. For the $\alpha \beta = 1$ case, we repeat the analysis of Section \ref{sec:neg} for the $\alpha \beta = 1$ case verbatim for $\tilde{\psi}$, which again allows us to conclude that $\psi - i \mathbf{H} \psi \equiv 0 \equiv \varphi + i \mathbf{H}\varphi$. By the considerations above, this implies the desired result also in that result. \\

Now, for Part (II) of Theorem \ref{thm:one-side}, we simply note that the functions $\psi_{\pm}$ built in Section \ref{sec:iteration-funct} satisfy exactly our assumptions, where we extend the sets $A$ and $B$ to the whole real line so that they satisfy the same uniform discreteness hypotheses. Hence, since $\psi_{\pm} \not\equiv 0$, it follows that the inclusion $\text{zclos}_2(\Lambda_{\bf A,B}) \subset \R_{\ge 0} \times \R_{\le 0}$ is strict, and since the functions $\psi_{\pm}$ are not identically zero, the complement of $\text{zclos}_2(\Lambda_{\bf A,B})$ in $\R_{\ge 0} \times \R_{\le 0}$ must necessarily contain infinitely many points, as desired. 
\end{proof}

\begin{remark} The argument above shows that the exact same result as Theorem \ref{thm:one-side} holds if $A \subset \R_{\le 0}, \, B \subset \R_{\ge 0}$, under the same hypotheses on the distribution of the difference of consecutive terms. 
\end{remark}

\subsection{The critical case of Theorem \ref{thm:improved}}\label{sec:alt-psi-proof} We now explain a different way to obtain the critical case $\alpha \beta = 1$ of Part (I) of Theorem \ref{thm:improved}, under several different sets of assumptions. We formulate it as a result of its own, which includes in its statement Theorem \ref{thm:L^2-HMR}. 

\begin{theorem} Let $A$ and $B$ be as in Theorem \ref{thm:improved}. The following assertions hold: 
\begin{enumerate}[\normalfont(I)]
    \item Let 
    \[
    \mathcal{C} = \left\{ \psi \in \mathcal{H}_0(\R) \colon \int_{\R} |\widehat{\psi}(x)| |x|^{1/4} \, dx < +\infty \right\}.
    \]
    Then, if $\alpha \beta = 1$, $(\Gamma,\Lambda_{\bf A,B})$ is a $\mathcal{C}$-\emph{H.U.P.}

    \vspace{2mm}
    
    \item If $A = \{\alpha n+\theta\}_{n \in \Z}$ and $B = \{\beta n\}_{n \in \Z}$, then $(\Gamma,\Lambda_{\bf A,B})$ is a $\mathcal{H}_0$-\emph{H.U.P.} if, and only if, $\alpha \beta \le 1$. 
\end{enumerate}
\end{theorem}
\vspace{-5mm}
\begin{proof}
Effectively, by the same argument as in the proof in Section \ref{sec:counterex}, we are able to conclude first that, in both cases, $\psi \in L^2((1+x^2)dx).$ Furthermore, we also have that $a_{n+1} - a_n = 1$ whenever $\widehat{\psi} \not\equiv 0$ on $[a_n,a_{n+1}]$. In that case, we may conclude that
\[
\widehat{\psi}(x) = t_n \sin\left( \pi \frac{x-a_n}{a_{n+1}-a_n}\right) = t_n \sin(\pi(x-a_n)) \text{ for } x \in [a_n,a_{n+1}).
\]
That is, 
\begin{equation}\label{eq:psi-hat-disj} 
\widehat{\psi}(x) = \sum_{n \in \Z} t_n \sin (\pi(x-a_n)) 1_{[a_n,a_{n+1}]}.
\end{equation} 
Taking inverse Fourier transforms on both sides, we obtain 
\[
\psi(t) = \sum_{n \in \Z}t_n e^{\pi i t(a_n+1/2)} \widecheck{\left(\cos(\pi (\cdot))1_{(-1/2,1/2)}\right)}(t).
\]
Since we may compute 
\begin{align*} 
\int_{-1/2}^{1/2} \cos(\pi x) e^{\pi i x t} \, dx &= \frac{1}{2i \pi} \left( \frac{e^{\frac{\pi}{2} i(t+1)} - e^{-\frac{\pi}{2} i (t+1)}}{t+1} + \frac{e^{\frac{\pi}{2}i(t-1)} - e^{-\frac{\pi}{2} i (t-1)}}{t-1} \right)  \cr 
    & = \frac{1}{\pi} \left( \frac{\sin\left( \frac{\pi}{2}(t+1)\right)}{t+1} + \frac{\sin\left( \frac{\pi}{2}(t-1)\right)}{t-1}\right) \cr 
    & = \frac{1}{\pi} \cos(\pi t/2) \left( \frac{1}{t+1} - \frac{1}{t-1}\right) = - \frac{2}{\pi} \frac{\cos(\pi t/2)}{t^2 - 1}, 
\end{align*} 
we obtain that 
\begin{equation}\label{eq:phi-one}
    \psi(t) = -\frac{1}{\sqrt{2} \pi(t^2 - 1)} \sum_{n \in \Z} t_n \left( e^{\pi i t a_{n+1}} + e^{\pi i t a_n}\right). 
\end{equation}
Since $\psi \in L^2(\R),$ \eqref{eq:psi-hat-disj} shows that $\{t_n\}_{n \in \Z} \in \ell^2(\Z).$ By the exact same argument applied to $\varphi(t) = t^{-2} \psi(1/t)$, we obtain that also 
\begin{equation}\label{eq:phi-two}
 t^{-2} \psi(1/t) = -\frac{1}{\sqrt{2} \pi(t^2-1)} \sum_{n \in \Z} r_n \left( e^{\pi i t b_{n+1}} + e^{\pi i t b_n} \right),
\end{equation}
with $\{r_n\}_{n \in \Z} \in \ell^2(\Z).$ By combining \eqref{eq:phi-one} and \eqref{eq:phi-two}, we get that 
\begin{equation}\label{eq:identity-poisson}
\sum_{n \in \Z} t_n \left( e^{\pi i t a_{n+1}} + e^{\pi i t a_n}\right) = - \sum_{n \in \Z} r_n \left( e^{\pi i b_{n+1}/t} + e^{\pi i b_n/t} \right), \, \forall \, t \in \R. 
\end{equation}
This last equation will be the main setup for our proofs. 
\vspace{2mm}

\noindent{\bf Part (I):} We first deal with the case of general $\mathcal{C}$-H.U.P.'s. Note that both sides in \eqref{eq:identity-poisson} define a tempered distribution. Hence, taking Fourier transforms on both sides implies that 
\begin{equation}\label{eq:fourier-poisson}
\sum_{ n \in \Z} t_n \left( \delta_{a_{n+1}} + \delta_{a_n}\right) = - \sum_{n \in \Z} r_n \left( \int_{\R} e^{-\pi i t \xi} e^{\pi i b_{n+1}/t} \, dt + \int_{\R} e^{-\pi i t \xi} e^{\pi i b_{n}/t} \, dt\right). 
\end{equation}
Here, the integrals on the right-hand side of \eqref{eq:fourier-poisson} have to be interpreted in the principal value sense. Now, using the results from \cite[Proposition~5.2.1]{Hedenmalm-Monte-Klein-Gordon}, such integrals may be evaluated as follows: 
\begin{equation}\label{eq:fourier-exp1/t}
\lim_{\varepsilon \to 0} \int_{\R} e^{-\pi i \left( t \xi - \eta/t\right)} e^{-\varepsilon |t|} \, dt = 2\pi \cdot \delta_0(|\eta|/|\xi|) - 2 \pi \sqrt{|\eta|/|\xi|} J_1\left(2\pi\sqrt{|\xi \eta|}\right) 1_{\{ \xi \eta < 0\}}. 
\end{equation}
Hence, rewriting \eqref{eq:fourier-poisson} in light of \eqref{eq:fourier-exp1/t}, 
\[
\sum_{ n \in \Z} t_n \left( \delta_{a_{n+1}} + \delta_{a_n}\right)
\]
\begin{equation}\label{eq:bessel-sums}
= \frac{2\pi}{\sqrt{|\xi|}} \sum_{n \in \Z} r_n \left( \sqrt{|b_n|} J_1\left(2\pi \sqrt{|\xi b_n|}\right)1_{\{b_n \xi < 0\}}  + \sqrt{|b_{n+1}|} J_1\left( 2\pi \sqrt{|\xi b_{n+1}|}\right)1_{\{b_{n+1} \xi<0\}}\right).
\end{equation}
Now, since $|J_1(s)| \le C/s^{1/2}$ holds for any $s > 0$, we conclude that, if we know that  
\begin{equation}\label{eq:coeff} 
\sum_{n \in \Z } |r_n| |b_n|^{1/4} < +\infty, 
\end{equation} 
then the right-hand side of \eqref{eq:bessel-sums} converges absolutely for any $\xi \colon |\xi| > 0$. This shows that $t_{n+1} = -t_n$ for all but at most one $n \in \Z$, which is a contradiction to the fact that $\{t_n\}_{n \in \Z} \in \ell^2(\Z),$ unless $\psi \equiv 0$, the desired conclusion which we desired to reach. Since \eqref{eq:coeff} follows if, for instance, 
\[
\int_{\R} |\widehat{\psi}(x)||x|^{1/4} \, dx < + \infty,
\]
we finish this part accordingly. 

\vspace{2mm}

\noindent{\bf Part (II):} We focus, in the remaining part, on the original case of translated lattice crosses considered in \cite{Hedenmalm-Montes-Annals, Hedenmalm-Monte-Klein-Gordon,Giri-Rawat,Giri-Rawat-corrigendum}. We rewrite \eqref{eq:identity-poisson} with $a_n = n+\theta, b_n = n$ as follows: 

\begin{equation}\label{eq:identity-poisson-integ}
e^{i\theta t} \sum_{n \in \Z} t_n \left( e^{\pi i n t} + e^{\pi i (n+1)t}\right) = - \sum_{n \in \Z} r_n \left( e^{\pi i (n+1)/t} + e^{\pi i n /t}\right). 
\end{equation}

Let the absolute value of the left-hand side of \eqref{eq:identity-poisson-integ} be denoted by $\Psi(t)$. This equation tells us thus that $\Psi$ satisfies 
\begin{equation}\label{eq:Psi-prop}
\begin{cases} \Psi(t+2) = \Psi(t) &\text{ for all } t \in \R,\cr 
\Psi\left(\frac{1}{t}\right) = \Psi\left(\frac{1}{t+2}\right)& \text{ for all } t \in \R\setminus\{0,-2\}.
\end{cases}
\end{equation} 
We further note that, by the properties of $\psi$, we must have necessarily $\Psi \in L^2(-1,1)$. We claim that these conditions imply that $\Psi$ is constant. 

Indeed, if not, then fix an arbitrary $\alpha > 0$ such that the set 
\[
E_{\alpha} = \left\{ t \in \R \colon \Psi(t) > \alpha \right\}
\]
satisfies $0<|E_{\alpha} \cap (-1,1)|<2$.  Consider now the function $1_{E_{\alpha}}$. By \eqref{eq:Psi-prop}, $t \in E_{\alpha}$ holds if and only if $t+2 \in E_{\alpha}$, and also $t \in E_{\alpha}$ if and only if $\frac{t}{1-2t} \in E_{\alpha}$. Hence, $1_{E_{\alpha}}$ satisfies the relationships in \eqref{eq:Psi-prop} as well. We now wish to contradict the fact that $0<|E_{\alpha} \cap (-1,1)|<2$, implying that either
$|E_{\alpha}| = 0$ or $|\R \setminus E_{\alpha}| = 0.$ Since $\alpha>0$ is arbitrary, we may conclude that $\Psi$ itself is constant. Hence, replacing $\Psi$ by $1_{E_{\alpha}}$ in the argument below if needed, we may suppose that $\Psi \in L^{\infty}(\R)$. 

The rest of the argument follows the same footsteps as in \cite[Section~5]{Hedenmalm-Montes-Annals}: indeed, while one could simply invoke Proposition 5.1 in \cite{Hedenmalm-Montes-Annals}, we chose to include the argument used for it below for completeness. 

Let then $\tilde{\Psi}:\C_+ \to \C$ denote the Poisson extension of $\Psi$ to the upper half plane. By definition $\tilde{\Psi}$ is a bounded harmonic function on the upper half plane, and a direct computation shows that 
\begin{equation}\label{eq:Psi-prop-ext}
\begin{cases} \tilde{\Psi}(z+2) = \tilde{\Psi}(z) &\text{ for all } z \in \C_+,\cr 
\tilde{\Psi}\left(z\right) = \tilde{\Psi}\left(\frac{z}{1-2z}\right)& \text{ for all } z \in \C_+.
\end{cases}
\end{equation} 
Let $\mathcal{G}$ denote the group of transformations preserving $\C_+$ generated by $z \mapsto z+2$ and $z \mapsto \frac{z}{1-2z}$. By for instance the results in \cite{Gilman-Maskit} and \cite{Beardon}, we deduce that $\mathcal{G}$ is a \emph{discrete and free group}, and a fundamental domain for $\C_+/\mathcal{G}$ is given by 
\[
\mathcal{D} = \left\{ z \in \C_+ \colon |\text{Re}(z)| < 1, \, \left| z -\frac{1}{2}\right| > \frac{1}{2} \text{ and } \left| z+\frac{1}{2}\right| > \frac{1}{2}\right\}. 
\]
It follows that $\tilde{\Psi}$ defines a bounded harmonic function on $\mathcal{D}$, and since the domain $\mathcal{D}$ only has cusps at $0,\infty$ and $\pm 1$, these are removable singularities of $\tilde{\Psi}$ when seen as a harmonic function on the quotient $\C_+ /\mathcal{G}$. Thus, $\tilde{\Psi}$ extends to a harmonic function on the compact Riemann surface induced by $\C_+ / \mathcal{G}$, which, by Liouville's theorem, implies $\tilde{\Psi}$ is constant, which yields in turn that $\Psi$ is constant. 

Finally, note that, given that $\Psi$ is constant, the only viable option is that $\Psi \equiv 0$, since turning to \eqref{eq:phi-one} shows that $\psi(t) = \frac{c}{t^2-1} \in L^2(\R)$, which can only hold if $c = 0$, as desired. 
\end{proof}

\begin{remark} Note that \eqref{eq:coeff} is fundamentally different from the $\mathcal{H}_2$ condition of Theorem \ref{thm:improved}. Indeed, 
the condition of belonging to the class $\mathcal{C}$ may be suitably translated into an assertion on \emph{regularity} of $\psi$. On the other hand, $\psi \in \mathcal{H}_2$ is an assertion on the \emph{decay} of $\psi$. We believe, in light of these comments, that a more careful analysis could lead to a sharper result on uniqueness of distributions solving \eqref{eq:identity-poisson} with square-summable coefficients, possibly by employing techniques from almost periodic functions, as done in \cite{Goncalves-Crystalline}. We remark that it has been recently brought to our attention that a similar argument has also been used in \cite{BHMRRV-second} in order to prove uniqueness of the so-called hyperbolic Fourier series expansion. 
\end{remark} 

{ 
\subsubsection{Proof of Theorem~\ref{thm:continuous-counterexample}.}
We first work with the cross \(\Lambda_{1,\beta}\). For a finite
measure \(\tau\) supported away from zero, we let in this section, for shortness,
\[
 H_n(\tau)=\int_{\R}e^{-\pi i nt}\,d\tau(t),
 \qquad
 V_n(\tau)=\int_{\R}e^{-\pi i\beta n/t}\,d\tau(t).
\]
We shall construct an even probability \(\nu\) for which
\begin{equation}\label{new-eq:c0-trace}
 H_n(\nu)=V_n(\nu)\qquad(n\in\Z).
\end{equation}
Indeed, choose a sufficiently large even integer \(L\) such that
\[
 L>\beta+2,
 \qquad \rho=\frac{\beta}{(L-1)^2}<\frac12,
\]
and define on \([-1,1]\)
\[
 s_+(t)=-\frac{\beta}{t+L},
 \qquad s_-(t)=-\frac{\beta}{t-L}.
\]
Their images are disjoint subsets of \((-1,1)\), and
\[
 0<s_\pm'(t)=\frac{\beta}{(t\pm L)^2}\leq\rho.
\]
For a probability measure \(\mu\) on \([-1,1]\), let
\[
 \mathcal P\mu=\frac12(s_+)_*\mu+\frac12(s_-)_*\mu.
\]
We use the \(1\)-Wasserstein distance on the set of Borel probability
measures on \([-1,1]\). Thus
\[
 W_1(\mu,\lambda)
 =\inf_{\pi\in\Pi(\mu,\lambda)}
   \int_{[-1,1]^2}|x-y|\,d\pi(x,y),
\]
where \(\Pi(\mu,\lambda)\) is the set of Borel probability measures on
\([-1,1]^2\) with first marginal \(\mu\) and second marginal
\(\lambda\). Fix \(\pi\in\Pi(\mu,\lambda)\), and define
\[
 \widetilde\pi
 =\frac12(s_+\times s_+)_*\pi
  +\frac12(s_-\times s_-)_*\pi.
\]
Its marginals are
\[
 (\operatorname{pr}_1)_*\widetilde\pi=\mathcal P\mu,
 \qquad
 (\operatorname{pr}_2)_*\widetilde\pi=\mathcal P\lambda,
\]
so \(\widetilde\pi\in\Pi(\mathcal P\mu,\mathcal P\lambda)\). The
mean-value theorem and the derivative bound above give
\[
 |s_\varepsilon(x)-s_\varepsilon(y)|
 \leq\rho|x-y|,
 \qquad \varepsilon\in\{+,-\}.
\]
Consequently,
\[
 \begin{aligned}
 W_1(\mathcal P\mu,\mathcal P\lambda)
 &\leq\int_{[-1,1]^2}|u-v|\,d\widetilde\pi(u,v)\\
 &=\frac12\sum_{\varepsilon\in\{+,-\}}
   \int_{[-1,1]^2}|s_\varepsilon(x)-s_\varepsilon(y)|\,d\pi(x,y)\\
 &\leq\rho\int_{[-1,1]^2}|x-y|\,d\pi(x,y).
 \end{aligned}
\]
Taking the infimum over \(\pi\in\Pi(\mu,\lambda)\) yields
\[
 W_1(\mathcal P\mu,\mathcal P\lambda)
 \leq\rho W_1(\mu,\lambda).
\]
The Borel probability measures on the compact interval \([-1,1]\)
form a complete metric space under \(W_1\). Since \(\rho<1\), the
contraction principle gives a unique fixed probability \(\nu\).
Moreover, reflection intertwines the two branches because
\(-s_+(-t)=s_-(t)\). Thus the reflection of \(\nu\) is also fixed by
\(\mathcal P\), and uniqueness shows that \(\nu\) is even.

For an even probability \(\mu\), the choice of \(L\) gives
\[
 \begin{aligned}
 V_n(\mathcal P\mu)
 &=\frac12\int\left(e^{-\pi i\beta n/s_+(t)}+
 e^{-\pi i\beta n/s_-(t)}\right)\,d\mu(t)\\
 &=\int e^{\pi i nt}\,d\mu(t)=H_n(\mu).
 \end{aligned}
\]
Indeed, \(\beta/s_+(t)=-t-L\),
\(\beta/s_-(t)=-t+L\), and \(L\) is even. Applying this identity to
the fixed point proves \eqref{new-eq:c0-trace}.

Let \(K=\operatorname{supp}\nu\). For a word
\(w=(\varepsilon_1,\ldots,\varepsilon_m)\in\{+,-\}^m\), write
\[
 s_w=s_{\varepsilon_1}\circ\cdots\circ s_{\varepsilon_m},
 \qquad I_w=s_w([-1,1]).
\]
Iterating the identity \(\mathcal P\nu=\nu\) gives
\begin{equation}\label{new-eq:ifs-iterate}
 \nu=2^{-m}\sum_{w\in\{+,-\}^m}(s_w)_*\nu.
\end{equation}
The first-level images are disjoint and both branches are injective.
Thus the intervals \(I_w\), with \(|w|=m\), are pairwise disjoint,
cover \(K\), and have diameter at most \(2\rho^m\). Equation
\eqref{new-eq:ifs-iterate} therefore gives \(\nu(I_w)=2^{-m}\). Hence
\(\nu(\{t\})\leq2^{-m}\) for every \(t\in K\) and every \(m\), so
\(\nu\) has no atoms. The same intervals cover \(K\) with total length
at most
\[
 2^m(2\rho^m)=2(2\rho)^m\longrightarrow0.
\]
Thus \(\mathcal L^1(K)=0\). It follows that \(\nu\) is singular
continuous.

Let \(\iota(t)=\beta/t\) and set
\begin{equation}\label{new-eq:c0-signed-measure}
 \sigma=\nu-\iota_*\nu.
\end{equation}
The support of \(\nu\) lies in
\[
 K\subset
 \left\{t:\frac{\beta}{L+1}\leq|t|\leq\frac{\beta}{L-1}\right\}
 \subset(-1,1).
\]
The identity \(K=s_+(K)\cup s_-(K)\), the evenness of \(K\), and the
formulas for \(s_\pm\) give
\[
 \iota(K)=(K-L)\cup(K+L).
\]
In particular, \(K\cap\iota(K)=\varnothing\), so the positive measures
\(\nu\) and \(\iota_*\nu\) are mutually singular. Hence
\[
 |\sigma|=\nu+\iota_*\nu,
 \qquad \operatorname{supp}\sigma=K\cup\iota(K),
\]
and \(\sigma\neq0\). Since \(K\) is compact and does not contain
zero, \(\iota\) is bi-Lipschitz on a neighborhood of \(K\). It follows
that \(\iota_*\nu\) is atomless and is supported on the null set
\(\iota(K)\). Thus \(|\sigma|\) is atomless and singular with respect
to Lebesgue measure. This is the sense in which the signed measure
\(\sigma\) is singular continuous. Finally, \(\nu\) is even and
\(\iota(-t)=-\iota(t)\), so \(\sigma\) is real and even. Since
\(\iota^2\) is the identity,
\[
 \iota_*\sigma=-\sigma.
\]
Equation~\eqref{new-eq:c0-trace} gives \(H_n(\sigma)=0\). The last
anti-invariance relation then gives \(V_n(\sigma)=0\) for every
\(n\in\Z\). It remains to prove the decay of the Fourier transform of \(\sigma\).

\begin{lemma}\label{new-lem:nonlinear-decay}
Let \(a\in C_c^1(\R)\), and let \(\phi\) be real and \(C^2\) on a
neighborhood of \(\operatorname{supp}a\). If \(|\phi'|\) is bounded
below by a positive constant on \(\operatorname{supp}a\), then
\[
 \int_K a(t)e^{ir\phi(t)}\,d\nu(t)\longrightarrow0
 \qquad(|r|\longrightarrow\infty).
\]
For the uniformity assertion, fix a compact set \(J\) contained in a
bounded open set \(U\subset\R\), with \(\operatorname{supp}a\subset J\)
and \(\phi\in C^2(U)\). The norms are \(\|a\|_{C^1(\R)}\) and
\(\|\phi\|_{C^2(U)}\).
The convergence is uniform when upper bounds for the two indicated norms and the
positive lower bound for \(|\phi'|\) are fixed.
\end{lemma}

\begin{proof}
The maps \(s_\pm\) are the projective actions of
\[
 g_\pm=\frac1{\sqrt\beta}
 \begin{pmatrix}0&-\beta\\1&\pm L\end{pmatrix}\in SL_2(\R).
\]
Put \(m=\tfrac12\delta_{g_+}+\tfrac12\delta_{g_-}\). This probability
measure has a finite exponential moment, since, for every \(\eta>0\),
\[
 \int\|g\|^\eta\,dm(g)
 =\tfrac12\|g_+\|^\eta+\tfrac12\|g_-\|^\eta<\infty.
\]
A direct computation gives
\[
 g_+^{-1}g_-=\begin{pmatrix}1&-2L\\0&1\end{pmatrix},
 \qquad
 g_+(g_+^{-1}g_-)g_+^{-1}
 =\begin{pmatrix}1&0\\2L/\beta&1\end{pmatrix}.
\]
The powers of these matrices show that the Zariski closure of the
generated group contains the upper and lower one-parameter unipotent
subgroups. Indeed, a polynomial restricted to either subgroup is a
polynomial in one real parameter, and if it vanishes on all integer powers of the corresponding nontrivial unipotent matrix, it vanishes identically.

It follows that these subgroups generate \(SL_2(\R)\), and hence the group is
Zariski dense. The identity \(\mathcal P\nu=\nu\) is precisely
\(m\)-stationarity on \(\mathbb P^1\), and since \(\mathcal P\nu=\nu\), the uniqueness of the
stationary measure identifies \(\nu\) with the measure in
\cite[Theorem~1.3]{Li-stationary}. Since the amplitude has compact
support in the affine chart, the amplitude and phase can be extended to
\(\mathbb P^1\) without changing them near
\(\operatorname{supp}a\). For uniformity, choose a fixed
\(\zeta\in C_c^\infty(U)\) equal to one near \(J\). Extend \(a\)
by zero at infinity and extend \(\zeta\phi\) by zero outside \(U\).
Their projective \(C^1\) and \(C^2\) norms are bounded in terms of
the two indicated norms and this fixed cutoff. On the fixed compact
set, the affine and projective coordinates change the derivative lower
bound only by fixed positive factors.
The cited theorem, including its uniformity
assertion, now gives the result.
\end{proof}

\begin{lemma}\label{new-lem:planar-decay}
If \(a\) is smooth near \(K\), then
\[
 \int_Ka(t)e^{-\pi i(xt+y/t)}\,d\nu(t)\longrightarrow0
 \quad\text{as }x^2+y^2\longrightarrow\infty.
\]
\end{lemma}

\begin{proof}
After multiplying \(a\) by a fixed smooth cutoff equal to one near
\(K\) and supported away from zero, we may suppose that
\(a\in C_c^\infty(\R\setminus\{0\})\). Write
\((x,y)=r\omega\), where \(r>0\) and \(|\omega|=1\), and set
\(\phi_\omega(t)=\omega_1t+\omega_2/t\). Suppose that a sequence of
these integrals does not tend to zero. After passing to a subsequence,
\(r\to\infty\) and \(\omega\to\omega_0\). The derivative
\(\phi_{\omega_0}'(t)=\omega_{0,1}-\omega_{0,2}/t^2\) has at most two
zeros on a compact neighborhood of \(K\). We take this neighborhood
to contain \(\operatorname{supp}a\) in its interior and to be disjoint
from zero. Since \(\nu\) has no atoms,
we may remove open neighborhoods of these zeros having arbitrarily
small \(\nu\)-mass. On the complement, choose a smooth cutoff \(\chi\)
which is one on \(K\) outside these neighborhoods. More precisely, for
a fixed \(\varepsilon>0\), choose the neighborhoods with total
\(\nu\)-mass less than \(\varepsilon\), and choose \(0\leq\chi\leq1\)
with compact support in the compact neighborhood above, vanishing on
smaller neighborhoods of all the stationary points. Thus
\(|\phi_{\omega_0}'|\) has a positive minimum on
\(\operatorname{supp}\chi\).
For all sufficiently
large indices, the absolute values of the derivatives of \(\phi_\omega\) are bounded below on
\(\operatorname{supp}\chi\), with uniform \(C^2\) bounds. Lemma~\ref{new-lem:nonlinear-decay} applies to \(a\chi\), uniformly
for this sequence of phases, and gives convergence to zero on the
remaining set. The omitted part is bounded by
\(\|a\|_{L^\infty(K)}\) times the arbitrarily small mass of the
neighborhoods. Thus the limit superior of the absolute values of the integrals
along this subsequence is at most
\(\|a\|_{L^\infty(K)}\varepsilon\). Here we first let \(r\to\infty\)
with the cutoff fixed, and then let \(\varepsilon\downarrow0\).
This is a contradiction.
\end{proof}

Define
\[
 v(x,y)=\int e^{-\pi i(xt+y/t)}\,d\nu(t),
 \qquad
 u(x,y)=\int e^{-\pi i(xt+y/t)}\,d\sigma(t).
\]
A change of variables in \eqref{new-eq:c0-signed-measure} gives
\begin{equation}\label{new-eq:c0-v-u}
 u(x,y)=v(x,y)-v(y/\beta,\beta x).
\end{equation}
Lemma~\ref{new-lem:planar-decay} and the invertibility of
\((x,y)\mapsto(y/\beta,\beta x)\) show that \(u\in C_0(\R^2)\).
The same lemma applies after every differentiation, since the resulting
amplitude for \(v\) is a constant multiple of \(t^{j-k}\), which is smooth on a
neighborhood of \(K\). For \(v(y/\beta,\beta x)\), the amplitude is
\((-\pi i)^{j+k}\beta^{j-k}t^{k-j}\), which is also smooth near \(K\).
Hence all derivatives of \(u\) lie in
\(C_0(\R^2)\). Compactness of the support away from zero justifies
differentiation under the integral and gives
\(u_{xy}+\pi^2u=0\). The measure \(\sigma\) is real and even, so
\(u\) is real-valued. The identities \(H_n(\sigma)=V_n(\sigma)=0\)
give the two families of lattice zeros. Under the Fourier normalization of this manuscript, the push-forward
of \(2\sigma\) under \(t\mapsto(t,t^{-1})\) has Fourier transform
\(u\). It is the spectral measure asserted in the theorem. Finally, \(u\not\equiv0\). Indeed, if its restriction to the
horizontal axis vanished, Fourier inversion against Schwartz test
functions would imply that \(\sigma\) annihilates every compactly
supported smooth function, which contradicts \(\sigma\ne0\).

It remains to transfer the singular-continuity assertion to the
spectral measure. Put \(G(t)=(t,t^{-1})\). On
\(\operatorname{supp}\sigma\), the map \(G\) is bi-Lipschitz: it is
Lipschitz because this compact set is bounded away from zero, and its
inverse is the projection onto the first coordinate. Hence
\(G_*(2\sigma)\) has no atoms. Its total variation is supported on
\(G(\operatorname{supp}\sigma)\), which has zero arclength because
\(\mathcal L^1(\operatorname{supp}\sigma)=0\). Thus the spectral
measure is singular continuous on \(\Gamma\).

For general \(\alpha,\beta>0\), perform the construction above with
vertical spacing \(\beta_0=\alpha\beta\), and denote the resulting
objects by \(u_0\) and \(\sigma_0\). Then
\[
 U(x,y)=u_0(x/\alpha,\alpha y)
 =\int e^{-\pi i(xs+y/s)}\,d\big((D_\alpha)_*\sigma_0\big)(s),
 \qquad D_\alpha(t)=t/\alpha.
\]
This change preserves the Klein--Gordon equation and transforms the
zero sets into \(\alpha\Z\times\{0\}\) and
\(\{0\}\times\beta\Z\). Since \(D_\alpha\) is bi-Lipschitz, the
spectral measure of \(U\) is still singular continuous.

Finally, choose sufficiently large even integers \(L_j\), each
satisfying the conditions above, with \(L_{j+1}\geq L_j+4\). The
intervals containing the corresponding sets \(K_j\) near zero are
pairwise disjoint, as are their translates by \(\pm L_j\). Thus the
measures \(\sigma_j\), and also their images under \(D_\alpha\) and
\(G\), have pairwise disjoint supports. They are linearly independent.
The injectivity of the Fourier transform on finite measures implies
that the associated solutions are linearly independent. This proves
the last assertion of the theorem.

If \(\varepsilon\ne0\), first carry out the construction with vertical
spacing \(\lvert\varepsilon\rvert\beta\). Replacing \(u(x,y)\) by
\(u(x,\varepsilon y)\) gives a corresponding example for
\(u_{xy}+\varepsilon\pi^2u=0\); the vertical zero set remains
\(\beta\Z\), since
\(\varepsilon\beta\Z=\lvert\varepsilon\rvert\beta\Z\).

\subsubsection{Proof of Theorem~\ref{prop:weighted-counterexample}.}
Fix \(\beta>1\), let \(b=\sqrt{\beta-1}\), and consider
\begin{equation}\label{new-eq:weighted-density}
 \psi(t)=\frac b\pi\left(
 \frac1{(t-1)^2+b^2}-\frac1{(t+1)^2+b^2}
 \right)
 =\frac{4bt}{\pi[t^4+2(\beta-2)t^2+\beta^2]}.
\end{equation}
The unnormalized Fourier transform of \(b/(\pi(t^2+b^2))\), with kernel
\(e^{-\pi ixt}\), is \(e^{-\pi b|x|}\). Translation consequently
shows that
\[
 \int_{\R}e^{-\pi ixt}\psi(t)\,dt
 =e^{-\pi b|x|}(e^{-\pi ix}-e^{\pi ix}),
\]
which vanishes for every integer \(x\). Directly from the rational
expression in \eqref{new-eq:weighted-density},
\[
 \frac\beta{t^2}\psi(\beta/t)=\psi(t)\qquad(t\ne0).
\]
Changing variables on each half-line therefore gives
\[
 \int_{\R}e^{-\pi i\beta n/t}\psi(t)\,dt
 =\int_{\R}e^{-\pi ins}\frac\beta{s^2}\psi(\beta/s)\,ds=0
 \qquad(n\in\Z).
\]

Let \(p'=p/(p-1)\). The function \(\psi\) is \(O(t)\) at zero and
\(O(|t|^{-3})\) at infinity. Hence
\[
 \int_{\R}|\psi(t)|^{p'}M(t)^{-1/(p-1)}\,dt
 =\int_{\R}|(1+t^2)\psi(t)|^{p'}\,dt<\infty.
\]
Thus \(F\mapsto\int F\psi\) is a continuous nonzero functional on
\(L^p(\R,M(t)\,dt)\), by H\"older's inequality. It annihilates both
families in the proposition, so their span cannot be dense. Finally,
\(M\in L^1\) and
\[
 \int_{\R}\frac{dt}
 {(1+t^2)^{p'}M(t)^{1/(p-1)}}
 =\int_{\R}1\,dt=\infty.
\]
This proves the proposition. For \(p=2\) and \(\beta=2\), the density
in \eqref{new-eq:weighted-density} is, up to a nonzero constant,
\(t/(t^4+4)\).
}

\section*{Acknowledgements}

We thank Mateus Sousa and Felipe Gon\c calves for several discussions that led to the investigation of the main results in this manuscript. 

D.R. acknowledges funding by the European Union (ERC, FourIntExP, 101078782). Views and opinions expressed are those of the author(s) only and do not necessarily reflect those of the European Union or European Research Council (ERC). Neither the European Union nor ERC can be held responsible for them. 

J. P. G. R.  was supported by the FCT through project SHADE
(2023.17881.ICDT), DOI \texttt{10.54499/}\allowbreak
\texttt{2023.17881.ICDT}. He was also supported  by FAPERJ through  JCNE grant
no.~SEI-260003/\allowbreak020475/2025, and by Instituto Serrapilheira through
grant Serra-R-2510-59765.

\nocite{Goncalves-Crystalline,BHMRRV-first}

\bibliography{perturbed-lattice-crosses-with-additions}
\bibliographystyle{amsplain}
\end{document}